\documentclass[a4paper]{article}
\usepackage{mathpazo}
\usepackage[english]{babel}
\usepackage[utf8]{inputenc}
\usepackage[T1]{fontenc}
\usepackage{amsfonts,amsmath,amssymb,amsthm,mathtools}
\usepackage[final]{graphicx}
\usepackage{tabularx}
\usepackage{subcaption}

\usepackage{color}
\usepackage{hyperref}
\usepackage{enumitem}
\usepackage{tikz}
\usepackage{tikz-cd}
\usetikzlibrary{arrows}
\usepackage{blkarray}
\usetikzlibrary{decorations.markings}

\newtheorem{theorem}{Theorem}[section]

\newtheorem{proposition}[theorem]{Proposition}
\newtheorem{lemma}[theorem]{Lemma}

\newtheorem*{corollary*}{Corollary}

\theoremstyle{definition}
\newtheorem{definition}[theorem]{Definition}
\newtheorem{example}[theorem]{Example}
\newtheorem{question}[theorem]{Question}

\theoremstyle{remark}
\newtheorem{remark}[theorem]{Remark}
\newtheorem{observation}[theorem]{Remark}
\newtheorem{notation}[theorem]{Notation}

\newcommand{\N}{\mathbf N}

\newcommand{\Z}{\mathbf Z}
\newcommand{\K}{\mathbf K}

\newcommand{\R}{\mathbf{R}}
\newcommand{\bR}{\mathbf{R}}

\newcommand{\GG}{\mathcal G}

\renewcommand{\L}{\mathcal{L}}

\newcommand{\im}{\text{im}}
\newcommand{\ra}{\rightarrow}

\usepackage{geometry}

\title{	Hochschild homology, and a persistent approach via connectivity digraphs}
\author{Luigi Caputi\footnote{School of Natural and Computing Sciences, University of Aberdeen, UK --
E-mail: luigi.caputi@abdn.ac.uk}, Henri Riihim\"{a}ki \footnote{Department of Mathematics, KTH, Stockholm, Sweden -- Email: henrir@kth.se}}

\begin{document}

\maketitle

\begin{abstract}
    We introduce a \emph{persistent Hochschild homology} framework for directed graphs. Hochschild homology groups of (path algebras of) directed graphs vanish in degree~$i\geq 2$. To  extend them to higher degrees,  we introduce the notion of \emph{connectivity digraphs} and analyse two main examples; the first, arising from  Atkin's \(q\)-connectivity, and the second, here called \emph{$n$-path digraphs}, generalising the classical notion of line graphs. Based on a categorical setting for persistent homology, we propose a stable pipeline for computing persistent Hochschild homology groups. This pipeline is also amenable to other homology theories; for this reason, we complement our work with a survey on homology theories of digraphs.
    \end{abstract}

Graph homology -- Persistent homology -- Hochschild homology -- Connectivity graphs -- \(n\)-path digraphs -- \(q\)-connectivity

\section*{Introduction}
Directed graphs, or digraphs, organise a multitude of mathematical objects and physical phenomena, in particular where an inherent directionality plays a considerable r\^ole. Prominent examples motivating this paper come from structural brain networks, \emph{i.e.},~(directed) networks modelling the synaptic connectivity in the brain; here, the pre- and post-synaptic signal propagation induces directions between neurons. This type of application in neuroscience has particularly ignited the interest in applied topology and topological data analysis (TDA), together with the subsequent development of computational tools~\cite{L_tgehetmann_2020,Reimann_2017}. For an application in classifying network (brain) dynamics, see the recent work~\cite{Tribes_math}. One of the main techniques adopted in TDA is persistent homology (PH), which has been employed not just in neuroscience and neuroimaging~\cite{neuro, chung, khalid, kuang, lee-discriminativePH,wong}, but also in fields like endoscopy analysis~\cite{DUNAEVA201613}, angiography~\cite{angio}, pulmonary diseases~\cite{Brodzki}, finance~\cite{gidea},  fingerprint classification~\cite{giansiracusa},  image classification~\cite{vmv.20171272}, to name a few. Based on a robust mathematical theory and being stable with respect to small noise  perturbations~\cite{stabilityph}, persistent homology has provided a novel qualitative and quantitative way to study complex data (in the form of either point clouds, time series or connectivity networks) via the associated (directed) graphs. 

 In concrete applications, the classical persistence homology  pipeline takes as input a filtered family of undirected graphs; in order to extend it to the directed framework, one needs suitable (co-)homology theories of directed graphs. 
In a primal approach, the pipeline would then be the following: one starts by constructing suitable simplicial complexes (the \emph{directed flag complex}~\cite{masulli-villa, Reimann_2017}) associated to the digraphs, computes (most often homological) features of the simplicial complexes, and finally uses these, or derived features, for subsequent network analysis. The implementation of the described pipeline is usually possible, thanks to existing softwares and algorithms allowing to carry on the homological computations; \emph{e.g.}, the software Flagser~\cite{L_tgehetmann_2020} has a persistence implementation with directed flag complexes (when the edges of the graphs are assigned some weighting). 
We review the construction of the directed flag complex  in Section~\ref{homologies}.  Even if the directed flag complexes are constructed using coherently oriented cliques in the digraphs, 
the computation of the associated homology groups reduces to simplicial homology. This has the effect that the associated topological invariants forget some  information carried in the directions of the edges; to see it, we present two basic homotopy equivalent spaces induced by different digraphs. Besides the directed flag complex approach, many other homology theories of digraphs have been used. In Section~\ref{homologies} we  review some of them. Among others, we give a review of the recently developed path homology~\cite{Grigoryan_first} -- see~Section~\ref{secpathhom} and \cite{chaplin2021betti} for a comparison with the directed flag complex of random graphs. We note here that the pitfall of path homology is that, beyond dimension~$2$, path homology groups  are  difficult to compute. 

A third remarkable approach uses Hochschild homology (\(HH\)), a homology theory of associative algebras introduced by Hochschild~\cite{MR11076} -- \emph{cf.}~Section~\ref{HH}. There is a standard and coherent way of associating to a directed graph an associative algebra, called the \emph{path algebra}. Then, application of Hochschild homology to path algebras of digraphs provides additional homological invariants. However, in Section~\ref{sec:whyt}, we show that all the described approaches cannot tell apart further simple examples of different digraphs.
This then begs the question about what an appropriate homology theory for digraphs should be, and how to incorporate the directed combinatorics in the theory and applications. The aim of this paper is to show that Hochschild homology of path algebras is able to capture part of the combinatorial information. Path algebras, in fact, naturally arise  from the combinatorics of the directed paths in the digraphs, and  Hochschild homology is able to coherently capture such information. If the digraph has an oriented cycle, the path algebra is automatically infinite-dimensional. In the case of acyclic digraphs instead, the dimensions of the Hochschild homology groups, seen as vector spaces over a base field, can be laid down in a simple combinatorial formula. However, also the Hochschild homology of digraphs has some shortcomings. First, the digraph is assumed to be acyclic, which rarely is the case in real networks. Second, the Hochschild homology groups vanish beyond degree~$1$, and one retrieves no information beyond dimension~$1$ of the digraph itself. 

To remedy the issues raised above, we propose in this paper the following approach: a persistent homology framework applied to a filtered family of digraphs taking into account higher orders of intrinsic connectivities. The main points of the paper follow the scheme below:

\textbf{From digraphs to connectivity digraphs.} 
With the twofold aim of providing and extending persistent homology pipelines  to higher degrees, and of capturing combinatorial information intrinsic in the directed structure of the digraphs, in this work we introduce the notion of a \emph{connectivity digraph} -- cf.~Definition~\ref{def:conn digraphs}. For a digraph $\GG$, a connectivity graph associated to~$\GG$ is a graph, possibly directed,  constructed by using the combinatorics of $\GG$. For example, connectivity graphs can be described by edges, paths, sets of edges, or more generally cliques, together with their incidence relations.
In Section~\ref{sec:connectivity_homology} we present two new connectivity structures for simplices. The \emph{directed \(q\)-analysis} extended from the work of Atkin \cite{Atkin1972} connects simplices \(\sigma\) and \(\tau\), both of dimension \(\geq q\), if the is a \(q\)-dimensional face \(\alpha\) such that \(\widehat{d_i}(\sigma) \hookleftarrow \alpha \hookrightarrow \widehat{d_j}(\tau)\), where \(\widehat{d_i}\) is an extended face map (see Definition \ref{def:di_hat}). The \emph{\(n\)-path digraph} connects \(n\)-simplices \(\sigma\) and \(\tau\) if there is an \((n-1)\)-simplex \(\alpha\) such that \(d_i(\sigma) = \alpha = d_j(\tau)\) with \(i < j\), where \(d_i\) is the standard face map. 
Both the above connectivity relations define connectivity digraphs.

{\textbf{Hochschild homology for acyclic graphs.} 
Application of homologies (simplicial homology, path homology, Hochschild homology) to connectivity digraphs  extends the family of homological invariants of digraphs to each $n\in \N$. In particular, applying Hochschild homology on connectivity digraphs enables us to admit Hochschild homology groups from degree 1 to \(n\), where \(n\) now refers to the dimension of simplices appearing in our connectivity digraphs. A  convenient computation of the dimension of Hochschild homology only applies when restricting to acyclic digraphs. In general, this fails to be true, also for our connectivity digraphs. To be able to compute Hochschild-related invariants for general digraphs, we define the \emph{Hochschild characteristic} (Definition~\ref{HHX}) which adds to the acyclic formula the component coming from the vector space generated by the simple cycles in the digraph.}

{\textbf{Stable persistent Hochschild homology.} In the case of acyclic digraphs the transformation from the category of digraphs, through connectivity structures, into finite Hochschild homology vector spaces is functorial. Then, for a given filtration of digraphs, one can apply the usual pipeline to get persistent Hochschild homology groups of digraphs. The categorical framework developed by  Bergomi and Vertechi~\cite{vertechi} provides immediately the needed abstract stability theorems.} For a filtration~$\mathcal{F}$ in the category~$\mathbf{Digraph}_0$ of digraphs without oriented cycles, the persistent Hochschild pipeline can be illustrated as the following composition of functors:
\[
(\bR, \leq) \xrightarrow{\mathcal{F}}\mathbf{Digraph}_0\xrightarrow{\mathcal{C}}\mathbf{Digraph}_0\xrightarrow{\K-} \K\text{-}\mathbf{Alg}\xrightarrow{\mathrm{HH}_1}\mathbf{FinVect} \ .
 \]
where $\mathcal{C}$ is any functorial construction of connectivity digraphs -- cf.~Section~\ref{sec:PHH} and \ref{secHHpipe}. 

\textbf{Computational \(HH\) and \(K\)-theory into applied topology toolbox.} As far as the authors know, this work is the first in bringing invariants from Hochschild homology into the realm of TDA, to be used as featurisations of common data objects. Natural extensions would be using cyclic homology theories and $K$-theoretic methods. Previous work has focused on computing the \(K\)-theory of the category of zig-zag persistence modules~\cite{GradySchecnfisch}. In a subsequent work we plan to investigate the extension in this $K$-theoretic directions as well. 

\vspace{0.2cm}

We hope with this work to raise the interest of current developments in applied topology and persistence to focus more on trying to bring unexplored tools from theoretical algebraic topology and \(K\)-theory into the applied setting. Based on our work, we believe that a following recipe is useful. First, a fruitful combination of data and invariants needs to be found. For us, this was the combinatorics of digraphs and the path algebras they generate. Standard simplicial persistence is very much made for finite metric spaces, which we feel is not a natural pairing for digraphs as combinatorial objects without any regard on metric issues. Second, the invariant needs to be computable. In our case, this comes from the known combinatorial formula for the dimensions of Hochschild homologies. Going further, one needs to focus algebraic derivations on proving similar results when these are not yet existing. Third, the main lesson from persistence theory are the stability results. Any new tool in the applied topology toolbox should take into account that small variations in the input data need to be bounded at the level of algebraic invariants and featurisations. As said above, our pipeline satisfies certain stability guarantees.

{We finish with some perspective on the apparent simplicity of digraphs. Even though graphs and digraphs are simple objects to describe and many concepts in graph theory are rather easy and intuitive to handle, it is the immense possibilities of putting together vertices and edges that gives rise to the actual complexity of graphs/digraphs. One then needs to find an appropriate balance between the objects described and complexity and information content of the invariants attached to them. As a fourth point in the recipe of the previous paragraph, in applications some level of interpretability of the invariants is desired. In standard persistence it is easy to give a geometric meaning to the generators of a barcode. In the case of Hochschild homology it is still maintainable to understand how the information of directed paths is captured. The question of interpretability is of course conditioned on the application domain. For example, in medical applications the outcomes of TDA analyses should have some meaning to, say, prognosis and diagnosis. In machine learning applications topological/algebraic invariants can be accepted more as a black box featurisations.}

\subsection*{Organisation of the paper}
We start in Section~\ref{sec:grandcpx} with recalling basic notions about directed graphs, simplicial complexes and homology theory. In Section~\ref{homologies}, we give an overview of the most employed homology theories of directed graphs: the homology of the directed flag complexes, path homology, Hochschild homology, and homology of categories. In Section~\ref{sec:connectivity_homology}, we introduce the connectivity digraphs. We analyse the \(q\)-connectivity digraphs first, and then the $n$-path digraphs. Connectivity digraphs are then used in Section~\ref{sec:extendedhomoloies} for the Hochschild homology pipeline.

\subsection*{Conventions}

Calligraphic font, as~$\GG$, is used to denote \emph{finite} graphs (both directed and undirected). 
All base rings are assumed to be unital and commutative, and algebras are assumed to be unital and associative. Unless otherwise stated, $R$ denotes a ring, $\K$ is an algebraically closed field, $A$ is a unital associative $R$-algebra, $V$ is a vector space over $\K$, and all tensor products  $\otimes$ are assumed to be over the base ring~$R$ or base field $\K$. 
General references for graph theory, category theory, and algebraic topology are \cite{West}, \cite{maclane:71}, and \cite{hatcher}, respectively.

\subsection*{Acknowledgements}
The authors  acknowledge support from the \'{E}cole Polytechnique F\'{e}d\'{e}rale de Lausanne via a collaboration agreement with the University of Aberdeen and wish to thank Ran Levi for his support and useful discussions. Henri Riihim\"{a}ki acknowledges partial support from the KTH Royal Institute of Technology in Stockholm.

\section{Graphs and complexes}\label{sec:grandcpx}
In this section we review and fix some basic notions related to graphs and simplicial complexes, needed in the follow-up. 

\subsection{The category of digraphs}

A \emph{graph} is pair \(\GG = (V,E)\) consisting of a set of vertices \(V\) and a relation \(E \subseteq [V]^2\). The relation \(E\) is the set of edges between vertices and we denote the edges by pairs \(\{v,w\}\). We are interested in graphs with oriented edges:

\begin{definition}\label{def:digraph}
	A \emph{directed graph}, or a \emph{digraph}, is pair \(\GG = (V,E)\) consisting of a set of vertices \(V\) and a subset \(E \subseteq (V \times V)/\Delta_V\), where \(\Delta_V = \{(v,v) \, | \, v \in V\}\). The subset \(E\) is the set of directed edges and we denote edges by ordered pairs \((v,w)\).
\end{definition}

In this work, unless otherwise specified, graphs and digraphs will always be finite, hence the sets $V$ and $E$ are  finite.
Note that the definition above defines simple (di)graphs without loops: there are no edges of the form \(\{v,v\}\) nor \((v,v)\) and there is only one edge between any pair of vertices. In the case of digraphs, edges are unique ordered pairs \((v,w)\), and we allow reciprocal edges \((v,w)\) and \((w,v)\) in \(E\). We use the same symbol $\GG$ for denoting both an (undirected) graph and a digraph. In the rest of the paper we will mainly deal with digraphs, and we will always make clear whether we are referring to a graph or a directed graph.

\begin{definition}\label{def:compl and cliques}
	A graph is \emph{complete} if for every pair of vertices \(v\) and \(w\) there is an edge \(\{v,w\}\). A digraph is complete if for every pair of vertices \(v\) and \(w\) there are both edges \((v,w)\) and \((w,v)\). A $k$-\emph{clique} of $\GG$ is a complete subgraph of $\GG$ on $k$ vertices.
\end{definition}

Directed graphs come equipped with \emph{source} and \emph{target} maps \(s,t\colon E\to V\). For an edge~\(e=(v,w)\), the function~\(s\) maps \(e\) to its source, \(s(e)=v\), and \(t\) to its target, \(t(e)=w\). Sometimes, when we want to specify the source and target maps, we denote a digraph as  \(\GG=(V,E,s,t)\). 

Graphs and digraphs have natural notions of morphisms between them; we spell it out in the case of digraphs. 

\begin{definition}\label{defmorgr}
	A \emph{morphism of digraphs} from  \(\GG_1 = (V_1,E_1)\) to  \(\GG_2 = (V_2,E_2)\) is a function $$\phi\colon V_1\to V_2$$ on the vertices such that $(\phi(v),\phi(w))\in E_2$ for every $(v,w)$ belonging to $ E_1$. \end{definition}

Observe that, by Definition~\ref{defmorgr}, a morphism of digraphs sends directed edges to directed  edges and it does not allow to collapse them. One can also consider  functions~$\phi\colon V_1\to V_2$ on the vertices  such that either $\phi(v)=\phi(w)$ or $(\phi(v),\phi(w))\in E_2$; we refer to these maps as \emph{maps of digraphs}. Finite digraphs and edge preserving morphisms of digraphs  form a category that we denote by $\mathbf{Digraph}$. 

\begin{observation}\label{obsmorgr}
	A morphism of digraphs from \(\GG_1\) to  \(\GG_2\) sends complete subgraphs of \(\GG_1\) to complete subgraphs of \(\GG_2\), hence cliques to cliques. Indeed, otherwise a morphism would collapse at least one of the edges in the clique, which is not allowed.
\end{observation}

One can consider also more restrictive morphisms, namely morphisms of digraphs that are also injective (as functions of vertices). We will refer to these morphisms as \emph{regular morphims} of digraphs and denote the resulting category of digraphs (possibly with loops) and regular morphisms by $\mathbf{RegDigraph}$.

\begin{observation}\label{rem:initobj}
	Both the categories  $\mathbf{Digraph}$ and $\mathbf{RegDigraph}$ have an initial object\footnote{An initial object  in a category $\mathbf{C}$ is an object $I$ such that, for each object 
	$C$ of $\mathbf{C}$, there is a unique morphism $I\to C$.}  $\emptyset$, the empty digraph. Note that this is not a terminal object.
\end{observation}

\begin{figure}[h]
	\centering
	\newdimen\R
	\R=2.0cm
	\begin{tikzpicture}
		\draw[xshift=5.0\R, fill] (270:\R) circle(.05)  node[below] {$v_n$};
		\draw[xshift=5.0\R,fill] (225:\R) circle(.05)  node[below left]   {$v_1$};
		\draw[xshift=5.0\R,fill] (180:\R) circle(.05)  node[left] {$v_2$};
		\draw[xshift=5.0\R,fill] (135:\R) circle(.05)  node[above left] {$v_3$};
		\draw[xshift=5.0\R, fill] (90:\R) circle(.05)  node[above] {$v_4$};
		\draw[xshift=5.0\R,fill] (45:\R) circle(.05)  node[above right] {$v_5$};
		\draw[xshift=5.0\R,fill] (0:\R) circle(.05)  node[right] {$v_6$};
		\draw[xshift=5.0\R,fill] (315:\R) circle(.05)  node[below right] {$v_{n-1}$};
		
		\node[xshift=5.0\R] (v0) at (270:\R) { };
		\node[xshift=5.0\R] (v1) at (225:\R) { };
		\node[xshift=5.0\R] (v2) at (180:\R) { };
		\node[xshift=5.0\R] (v3) at (135:\R) { };
		\node[xshift=5.0\R] (v4) at (90:\R) { };
		\node[xshift=5.0\R] (v5) at (45:\R) { };
		\node[xshift=5.0\R] (v6) at (0:\R) { };
		\node[xshift=5.0\R] (vn) at (315:\R) { };
		
		\draw[thick, red, -latex] (v0)--(v1);
		\draw[thick, red, -latex] (v1)--(v2);
		\draw[thick, red, -latex] (v2)--(v3);
		\draw[thick, red, -latex] (v3)--(v4);
		\draw[thick, red, -latex] (v4)--(v5);
		\draw[thick, red, -latex] (v5)--(v6);
		\draw[thick, red, -latex] (vn)--(v0);

		\draw[xshift=5.0\R, fill] (292.5:\R) node[below right] {$e_{n-1}$};
		\draw[xshift=5.0\R,fill] (247.5:\R) node[below left] {$e_n$};
		\draw[xshift=5.0\R,fill] (202.5:\R)   node[left] {$e_1$};
		\draw[xshift=5.0\R,fill] (157.5:\R)  node[above left] {$e_2$};
		\draw[xshift=5.0\R, fill] (112.5:\R)   node[above] {$e_3$};
		\draw[xshift=5.0\R,fill] (67.5:\R) node[above right] {$e_4$};
		\draw[xshift=5.0\R,fill] (22.5:\R) node[right] {$e_5$};
		\draw[xshift=4.95\R,fill] (337.5:\R)  node {$\cdot$} ;
		\draw[xshift=4.95\R,fill] (333:\R)  node {$\cdot$} ;
		\draw[xshift=4.95\R,fill] (342:\R)  node {$\cdot$} ;
	\end{tikzpicture}
	\caption{The coherently oriented cyclic digraph $C_n$.} 
	\label{fig:poly}
\end{figure}
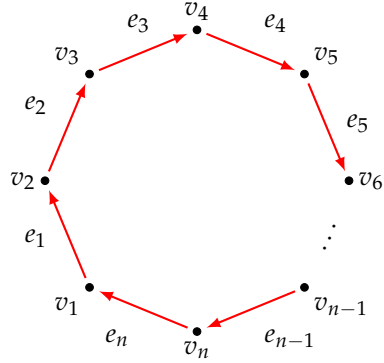

An oriented cycle in a directed graph $\GG$ is an embedding into $\GG$ of the coherently oriented cyclic digraph $C_n$ on \(n\) vertices -- \emph{cf.}~Figure~\ref{fig:poly}. In the follow-up, we might need to work with categories of digraphs without oriented cycles; we use the following notation:

\begin{notation}
We denote by  $\mathbf{Digraph_0}$  the subcategory of $\mathbf{Digraph}$  consisting of finite directed graphs without oriented cycles.
\end{notation}

A standard construction in graph theory is  the so called \emph{edge graph}, or \emph{line graph}, $\L(\GG)$ of a graph $\GG$. This is defined as the graph consisting of all the edges of $\GG$ as vertices, with connections described by the incidence relations. The construction generalizes to the case of digraphs -- see, for instance, \cite{Harary1960SomePO}:

\begin{definition}\leavevmode\label{deflinegr}
	\label{deflinedigr} The \emph{line digraph} of a directed graph $\GG=(V,E,s,t)$ is the directed graph $\L(\GG)$ whose vertices are the edges of $\GG$ and two vertices $p$ and $q$ in $\L(\GG)$ corresponding to the edges $e_p=(s(e_p), t(e_p))$ and $e_q=(s(e_q),t(e_q))$ in $\GG$ are connected by a directed edge $(p,q)\in E(\L(\GG))$ if $t(e_p)=s(e_q)$. 
\end{definition}

Associating a line (di)graph to a (di)graph is coherent with respect to morphisms:

\begin{observation}\label{linefunctor}
There is a functor 
	$$\mathcal{L}\colon \mathbf{Digraph}\to\mathbf{Digraph}$$ which sends a digraph $\GG$ to its line digraph $\L(\GG)$. 
First, note that the line digraph  $\L(\GG)$ of  a digraph~$\GG$ is the empty digraph $\emptyset$ if, and only if, $\GG$ has no edges. As by Remark~\ref{rem:initobj}, $\emptyset$ is an initial but not a terminal object in the category $\mathbf{Digraph}$, the functoriality may fail for morphisms $\phi\colon \GG_1\to \GG_2$, where $\GG_2$ has no edges. On the other hand, a morphism of digraphs sends edges to edges and collapsing is not allowed; therefore, either both $\GG_1$ and $\GG_2$ have no edges -- hence, the induced morphism between the associated line digraphs is the trivial morphism $\emptyset\to \emptyset$ -- or $\phi$ induces a morphism of digraphs $\L(\phi)\colon\L(\GG_1)\to\L(\GG_2)$ between the associated (non-empty) line digraphs. It is now straightforward to check that compositions and identities are preserved, hence $\mathcal{L}$ is a functor. 	
\end{observation}

{We will use a standard procedure for obtaining directed acyclic graphs out of a digraph.
\begin{definition}
	A \emph{strongly connected component} in a digraph \(\GG\) is an induced subgraph \(\GG'\) such that for any two vertices \(x\) and \(y\) in \(\GG'\) there are paths \(x \ra y\) and \(y \ra x\) in \(\GG'\).
\end{definition}
The strongly connected components are the equivalence classes of the relation of being strongly connected on the vertices of \(\GG\), i.e.~having directed paths between any ordered pair of vertices. The ensuing partition then enables to construct the quotient graph without directed cycles.
\begin{definition}\label{def:condensation}
	The \emph{condensation} \(c(\GG)\) of digraph \(\GG\) has as its vertices the strongly connected components of \(\GG\). Two vertices \(X\) and \(Y\) have a directed edge \((X,Y)\) in \(c(\GG)\) if there is an edge \((x,y)\) in \(\GG\) for some \(x \in X\) and \(y \in Y\).
\end{definition}

\begin{remark}\label{rem:condensation wo cycles}
	The {condensation} \(c(\GG)\) of a digraph \(\GG\) does not have oriented cycles.
\end{remark}
In particular, if \(\GG\) has the structure of a preorder, i.e.~a reflexive and transitive relation, the condensation is a canonical way of obtaining its underlying partial order \cite[Proposition 8.13]{Schroeder_book}.} Observe that taking the condensation of a digraph is not functorial; in fact, it may lead to maps $G\to *$, where $*$ is the one-point graph. However, in the theory of Alexandroff and finite topological spaces preorders and partial orders are in bijection with topological spaces and spaces with \(T_0\) separation, respectively. In this context, the condensation is a homotopy equivalence~\cite{Barmak_book}.

\subsection{Simplicial complexes and homology theories}\label{sec:scompandhom}

We recall here the definition of simplicial complexes;

\begin{definition}\label{def:abstract_simplicial complex}
	An \emph{(abstract) simplicial complex} on a vertex set \(V\) is a collection \(K\) of non-empty finite subsets \(\sigma \subseteq V\) that is closed under taking non-empty subsets: if \(\sigma \in K\) and \(\tau \subseteq \sigma\) is non-empty then \(\tau \in K\). The subsets are called simplices of \(K\).  
\end{definition}

The following list records notations related to simplices and simplicial complexes used in this paper.
\begin{center}
	\begin{tabular}{ c | p{8.5cm} }
		\hline
		Notation & Definition\\ \hline
		\(\sigma \in K\) & \(\sigma\) is a simplex in a simplicial complex \(K\). \\ \hline
		\(K_q\) & the set of simplices of \(K\) with dimension greater than or equal to \(q\). \\ \hline
		\(\text{Vert}(K)\) or $V(K)$, \(\text{Vert}(\sigma)\) & The sets of vertices of \(K\) and \(\sigma\), respectively. \\ \hline
		\(\text{dim}(\sigma)\) & \(|\text{Vert}(\sigma)| - 1\), dimension of \(\sigma\). If equal to \(k\), then \(\sigma\) is a \(k\)-simplex. \\ \hline
		\(\text{dim}(K)\) & The dimension of \(K\), the dimension of its highest dimensional simplex. \\ \hline
		\(\tau \subseteq \sigma\), \(\tau \hookrightarrow \sigma\) & Face of \(\sigma\). Faces are simplices. We use the convention that every simplex is a face of itself. Proper face has dimension strictly less than the dimension of the simplex.		
	\end{tabular}
\end{center}

Analogously to morphisms of graphs, we can define morphisms of simplicial complexes:

\begin{definition}
	A \emph{simplicial map} $f\colon K_1\to K_2$ between simplicial complexes $K_1$ and $K_2$ is a function on the vertices $f\colon V(K_1)\to V(K_2)$ such that $f(\sigma)\in K_2$ is a simplex, for every simplex $\sigma$ of $K_1$.
\end{definition}

(Abstract) simplicial complexes and simplicial maps form the category of simplicial complexes, that we denote by $\mathbf{SCpx}$. In this work we focus on homological invariants of directed graphs, and more precisely, on homology groups of digraphs. Homology groups are topological invariants of simplicial complexes. We assume that the reader is familiar with homology theories and we refer to \cite{hatcher, munkres} for comprehensive introductions. For setting the notations we briefly sketch here the main definition. We fix a commutative ring $R$.

\begin{definition}
	A \emph{chain complex} $(C,\partial)$ is a sequence $C=(C_n)_{n\in \N}$ of $R$-modules  with a boundary operator $\partial$  consisting of linear maps $\partial=(\partial_n\colon C_{n+1}\to C_{n})$ such that $\partial_{n}\circ \partial_{n+1}=0$ for all~$n$.  
\end{definition}

A \emph{morphism of chain complexes} $f\colon (C,\partial)\to (C',\partial')$ is a sequence of linear maps $f_n\colon C_n\to C_n'$ with the commutation relations $f_n \circ\partial_n = \partial_n' \circ f_{n+1}$. Chain complexes and morphisms of chain complexes over $R$ form a category~$\mathbf{Ch}(R)$, or more concisely $\mathbf{Ch}$.
For  a simplicial complex~$K$ and  a commutative ring $R$, there is a standard way to construct a chain complex~$(C,\partial)$ by considering, for each $n\in \N$,  the free $R$-module generated by the $n$-simplices of $K$. The construction gives a functor from the category~$\mathbf{SCpx}$ of simpicial complexes to the category~$\mathbf{Ch}$. Furthermore, for a chain complex $(C,\partial)$, the \emph{degree \(n\) homology} $\mathrm{H}_n(C)$ of $(C,\partial)$ is defined as the quotient
\[
\mathrm{H}_n(C)\coloneqq \ker(\partial_n)/\im(\partial_{n+1}),
\]
which is well-defined as $\im(\partial_{n+1})$ is contained in $\ker(\partial_n)$ by the identity $\partial_{n}\circ \partial_{n+1}=0$. If we want to indicate the coefficients \(R\) over which we are computing homology we write \(\mathrm{H}_n(C;R)\). The construction  gives functors from the category of chain complexes~$\mathbf{Ch}$ to the category~$\mathbf{Mod}_R$ of $R$-modules; hence, by composition, funtors from the category~$\mathbf{SCpx}$ to $\mathbf{Mod}_R$. In the next section we briefly recall how to construct, starting with a digraph, suitable simplicial complexes (flag complexes or path complexes) and  homology groups of digraphs.

\section{Homology theories of digraphs}\label{homologies}
In this section we survey some of the most prominent  homology theories of directed graphs. We start in Section~\ref{secflcompl} by recalling the definition of flag complexes and associated simplicial homology, then in Section~\ref{secpathhom} we review the path complexes and path homology, as introduced by Grigor’yan \textit{et al.}~\cite{Grigoryan_first}. In Section~\ref{secpathhom}, we provide a more detailed account on the Hochschild homology of a digraph, as of more relevance to us. Finally, in Section~\ref{sechomothers}, we sketch some  variations to these constructions, as has appeared in the literature. 

\subsection{Homology of flag complexes}\label{secflcompl}
Homology groups of undirected graphs can be defined as the simplicial homology groups of their underlying topological spaces; in fact, graphs can be seen as $1$-dimensional simplicial complexes, and one can directly apply the approach of Section~\ref{sec:scompandhom}. We first start with describing this naive approach, and then we see how  to generalize it by means of the so-called flag complexes.

For a digraph $\GG=(E,V)$ and a fixed commutative ring $R$, consider the chain complex
\begin{equation}\label{ordhomgr}
\cdots\xrightarrow{0}0\xrightarrow{0} \langle E\rangle_R\xrightarrow{\partial_1}\langle V\rangle_R\xrightarrow{0}0
\end{equation}
where $\langle E\rangle_R$ is the free $R$-module generated by the edges $E$ and $\langle V\rangle_R$ is the free $R$-module generated by the vertices $V$ of $\GG$. The boundary maps $\partial_i$ are all $0$, except for $\partial_1$; this acts on the basis edges in $\langle E\rangle_R$ as
\[
\partial_1(v,w)\coloneqq w-v,
\]
and is extended to the whole $R$-module $\langle E\rangle_R$ by $R$-linearity. The homology groups of a digraph defined this way are usually referred to  as the \emph{ordinary homology} groups of graphs. 

\begin{remark}
    The ordinary homology groups of graphs are trivial in every degree $i\geq 2$. 
\end{remark}

The $0$-th homology group of $\GG$ describes the set of connected components. The $1$-st homology group is isomorphic to the kernel of the only non-trivial map $\partial_1$, and counts the cycles of $\GG$; its rank can be entirely described in terms of numbers of vertices, edges and connected components of the digraph $\GG$ -- see for instance \cite[Theorem~1.9.6]{grdiestel}. 
A prominent approach to generate higher dimensional homology groups is to construct, out of a graph~$\GG$, the so-called flag (also known as clique) complexes~\cite{Aharoni,CHEN2001153,IVASHCHENKO1994159}; these provide natural invariants of graphs and have been generalized to digraphs~\cite{masulli-villa, Reimann_2017}, as we now recall.

We first need to introduce the ordered simplicial complexes. A set $S$, endowed with a linear order of its elements, will be called an \emph{ordered set}. 

\begin{definition}\label{def:osc}
An \emph{ordered simplicial complex} $\Sigma$  on a vertex set $V$ is a non-empty family of finite ordered subsets $\sigma\subseteq V$ with the property that, if $\sigma$ belongs to $\Sigma$ then every ordered subset $\tau$ of~$\sigma$ (ordered with the natural order induced by $\sigma$) belongs to $\Sigma$.
\end{definition}

When dealing with directed graphs, we need ordered cliques (as opposed to unordered cliques -- cf.~Definition~\ref{def:compl and cliques}.

\begin{definition}
     An \emph{ordered $k$-clique} of a directed graph $\GG$ is a totally ordered $k$-tuple $(v_1,...,v_k)$ of vertices of~$\GG$ with the property that, for every $i < j$, the pair $(v_i,v_j)$ is an ordered edge of $\GG$.
\end{definition}

 We can now extend the construction of flag complexes to directed graphs. 
 
 \begin{definition}\label{defFl}
 	Let $\GG=(V,E)$ be a directed graph. The \emph{directed flag complex} of $\GG$ is the ordered simplicial complex $\mathrm{dFl}(\GG)$ on $V$ whose $k$-simplices are all the ordered $(k+1)$-cliques of $\GG$.
 \end{definition}

We can again construct a chain complex. Let  $C_n(\mathrm{dFl}(\GG);R)$ be the $R$-module freely generated by all the $n$-simplices of $\mathrm{dFl}(\GG)$. There are well-defined face maps 
\begin{equation}
    d_j\colon C_n(\mathrm{dFl}(\GG);R)\to C_{n-1}(\mathrm{dFl}(\GG);R)
\end{equation}
for $j=0,\dots,n$. The $j$-th face map $d_j$, as an operator applied to the simplex $[v_0,\dots,v_n]$ in $C_n(\mathrm{dFl}(\GG);R)$ is defined by cancelling the $j$-th vertex: $$d_j[v_0,\dots,v_n]\coloneqq [v_0,\dots,\widehat{v_j}, \dots, v_n].$$ 
Face maps uniquely identify the faces of a simplex. In fact, let  $\sigma=[v_0,\dots,v_n]$ be an $n$-simplex  of the flag complex $\mathrm{dFl}(\GG)$; then, the  faces $d_j[v_0,\dots,v_n]$ are $(n-1)$-simplices of $\mathrm{dFl}(\GG)$. Each face map $d_j$ uniquely identifies the $j$-th $(n-1)$-face of $\sigma$ as the face opposite to the vertex $v_j$.
	For example, if $[v_0,v_1,v_2]$ is an ordered $3$-clique in a digraph $\GG$, represented below as an ordered simplex $\sigma$ of the associated directed flag complex $\mathrm{dFl}(\GG)$,
	\begin{center}
		\begin{tikzpicture}[baseline=(current bounding box.center)]
			\tikzstyle{point}=[circle,thick,draw=black,fill=black,inner sep=0pt,minimum width=2pt,minimum height=2pt]
		\tikzstyle{arc}=[shorten >= 8pt,shorten <= 8pt,->, thick]
		
		\node[] (0) at (-1.2,-1) {$v_0$};
		\node[] (2) at (1.2,-1) {$v_2$};
		\node[] (1) at (0,1) {$v_1$};
		
		\draw[arc] (-1.2,-1) to node[label=  left:$d_2(\sigma)$] {} (0,1) ;
		\draw[arc] (-1.2,-1) to node[label=  below:$d_1(\sigma)$] {} (1.2,-1);
		\draw[arc] (0,1) to node[label=  right:$d_0(\sigma)$] {} (1.2,-1);
		\end{tikzpicture}
	\end{center}
	 then we have $d_0(\sigma)=[v_1,v_2]$, $d_1(\sigma)=[v_0,v_2]$ and $d_2(\sigma)=[v_0,v_1]$.

\begin{remark}\label{rem:digraph_morphism_on_cliques}
The construction of flag complexes can be promoted to a functor from directed graphs to ordered simplicial complexes. Namely, if $\phi\colon \GG_1\to\GG_2$ is a morphism of digraphs, then it sends ordered cliques of $\GG_1$ to ordered cliques of $\GG_2$. This induces a simplicial morphism $f_{\phi}\colon \mathrm{dFl}(\GG_1)\to \mathrm{dFl}(\GG_2)$ between the flag complexes, sending a simplex $\sigma\in \mathrm{dFl}(\GG_1)$, hence an ordered clique $(v_0,\dots,v_k)$ of $\GG_1$, to the simplex $f_{\phi}(\sigma)=(f(v_0),\dots,f(v_k))$. 
\end{remark}

The (simplicial) homology groups with coefficients in \(R\) of a directed graph $\GG$ are defined as the  homology groups of the associated directed flag complex, and for each $n\in \N$ it can be seen as a composition of functors
\begin{equation}\label{eq:dFl}
\mathbf{Digraph} \xrightarrow{\mathrm{Ch}\circ\mathrm{dFl}} 
\mathbf{Ch}\xrightarrow{\mathrm{H}_n(-;R)} \mathbf{Mod}_R,
\end{equation}
where~$\mathbf{Mod}_R$ is the category  of $R$-modules.

\begin{example}\label{ex:square}
Fix as base ring the ring of integers $\Z$. 
 	Consider the square-shaped digraph $\GG$ on the vertices $0,1,2,3$ and set of edges described as in the following picture:
\begin{center}
\begin{tikzpicture}
	\tikzstyle{arc}=[shorten >= 3pt,shorten <= 3pt,->, thick]
	
	\node[] (0) at  (0,0) {0};
	\node[] (1) at  (0,1.5) {1};
	\node[] (2) at  (1.5,0) {2};
	\node[] (3) at  (1.5,1.5) {$3$};
	
	\draw[arc, red] (0) to (1);
	\draw[arc, red] (0) to (2);
	\draw[arc, red] (1) to (3);
	\draw[arc, red] (2) to (3);
	
	\node at (-0.65,0.75) {\(\GG=\)};
\tikzstyle{point}=[circle,thick,draw=black,fill=black,inner sep=0pt,minimum width=2pt,minimum height=2pt]
\node[point] at (4,0) {};
\node[point] at (5.5,0) {};
\node[point] at (4,1.5) {};	
\node[point] at (5.5,1.5) {};

\draw (4,0)  -- (5.5,0);		
\draw (4,0)  -- (4,1.5);
\draw (5.5,0)  -- (5.5,1.5);
\draw (4,1.5)  -- (5.5,1.5);

\node at (2.85,0.75) {\(\Rightarrow \ \mathrm{dFl}(\GG)=\)};						
\end{tikzpicture}
\end{center}
As we only have  ordered 2-cliques, the associated flag complex is the square itself, i.e.~an ordered simplicial complex with four vertices and four edges. It has homology groups  $$\mathrm{H}_0(\mathrm{dFl}(\GG);\Z)\cong \Z\cong\mathrm{H}_1(\mathrm{dFl}(\GG);\Z).$$ By adding to $\GG$ also the edge $(0,3)$, we get the ordered $3$-cliques $(0,1,3)$ and $(0,2,3)$; the corresponding flag complex is then the full square:
\begin{center}
	\begin{tikzpicture}
	\tikzstyle{arc}=[shorten >= 3pt,shorten <= 3pt,->, thick]
	
	\node[] (0) at  (0,0) {0};
	\node[] (1) at  (0,1.5) {1};
	\node[] (2) at  (1.5,0) {2};
	\node[] (3) at  (1.5,1.5) {$3$};
	
	\draw[arc, red] (0) to (1);
	\draw[arc, red] (0) to (2);
	\draw[arc, red] (1) to (3);
	\draw[arc, red] (2) to (3);
	\draw[arc, red] (0) to (3);
	
	\node at (-0.65,0.75) {\(\GG=\)};
	\tikzstyle{point}=[circle,thick,draw=black,fill=black,inner sep=0pt,minimum width=2pt,minimum height=2pt]
	\node[point] at (4,0) {};
	\node[point] at (5.5,0) {};
	\node[point] at (4,1.5) {};	
	\node[point] at (5.5,1.5) {};
	
	\draw[fill=green,opacity=0.25] (4,0) -- (5.5,0) -- (5.5,1.5) -- (4,1.5) --cycle; 
	\draw (4,0)  -- (5.5,0);		
	\draw (4,0)  -- (4,1.5);
	\draw (5.5,0)  -- (5.5,1.5);
	\draw (4,1.5)  -- (5.5,1.5);
	\draw (4,0)  -- (5.5,1.5);
	\draw[fill=green,opacity=0.25] (4,0) -- (5.5,0) -- (5.5,1.5) -- (4,1.5) --cycle; 
	
	\node at (2.85,0.75) {\(\Rightarrow \ \mathrm{dFl}(\GG)=\)};						
	\end{tikzpicture}
\end{center}
The first homology group $\mathrm{H}_1(\mathrm{dFl}(\GG);\Z)$ is now $0$, whereas the $0$-th homology group describing the number of connected components is again isomorphic to~$\Z$.
\end{example}

\begin{observation}\label{obs:flagsforgetdir}
From Example~\ref{ex:square} we see that some information from the orientations of the edges is lost. 
In fact, even though the simplices are constructed from the ordered cliques, the homology groups are defined as the homology groups of the (geometric realization of the) directed flag complex, which forgets some information about the directionalities. 
As an additional illustrative example, 
consider the digraphs in Figure~\ref{fig:possiblegraphs}. They have isomorphic ordinary homology groups, as well as isomorphic homology groups of the associated directed flag complexes. 
\begin{figure}[h]
    \centering
    \begin{tikzpicture}
	\tikzstyle{point}=[circle,thick,draw=black,fill=black,inner sep=0pt,minimum width=2pt,minimum height=2pt]
	\tikzstyle{arc}=[shorten >= 8pt,shorten <= 8pt,->, thick]
	
	\node[] at (-1.2,-1) {1};
	\node[] at (-1.6,0.5) {3};
	\node[] at (1.6,0.5) {4};
	\node[] at (1.2,-1) {2};
	\node[] at (0,1) {0};
	
	\draw[arc, red] (0,1) to (-1.6,0.5);
	\draw[arc, red] (-1.6,0.5) to (-1.2,-1);
	\draw[arc, red] (-1.2,-1) to (1.2,-1);
	\draw[arc, red] (1.6,0.5) to (0,1);
	\draw[arc, red] (1.2,-1) to (1.6,0.5);
	
	\node[] at (3.0,-1) {1};
	\node[] at (2.6,0.5) {3};
	\node[] at (5.4,-1) {2};
	\node[] at (5.8,0.5) {4};
	\node[] at (4.2,1) {0};
	
	\draw[arc, red] (4.2,1) to (2.6,0.5);
	\draw[arc, red] (4.2,1) to (5.8,0.5);
	\draw[arc, red] (3.0,-1) to (5.4,-1);
	\draw[arc, red] (2.6,0.5) to (3.0,-1);
	\draw[arc, red] (5.4,-1) to (5.8,0.5);
	\end{tikzpicture}
	\caption{Two possible directed configurations for pentagon graph.}
    \label{fig:possiblegraphs}
	\end{figure}
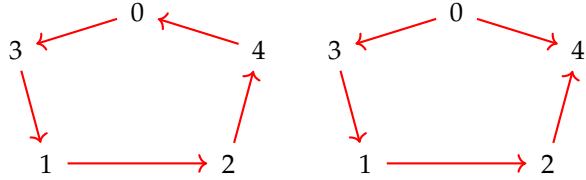
\end{observation}

Therefore, for constructing homology theories of digraphs more sensitive to the directionalities, one might need to incorporate the directed combinatorics in the definition of the homology groups. This is partially achieved with the homology theories recalled in the next subsections. 

\subsection{Path homology}\label{secpathhom}

Path homology can be considered as a homology theory of directed graphs, explicitly constructed from the edges of the digraphs. It was introduced in \cite{grigoryan2013homologies,Grigoryan_first}, and it has nowadays many developments. We recall its definition, following the first works on the subject.

\begin{definition}
Let $V$ be a finite set. For $p\in \N$, an \emph{elementary $p$-path} on $V$ is an ordered sequence $i_0\dots i_p$ of $p+1$ elements of $V$.
\end{definition}

Let $\K$ be a field (this assumption may be relaxed, but for the sake of reference, we use here the same assumptions as in \cite{grigoryan2013homologies}).
The vector space over $\K$ consisting of formal linear combinations of elementary $p$-paths is denoted by $\Lambda_p(V)$, or simply by $\Lambda_p$. The basis element of $\Lambda_p(V)$ corresponding  to the elementary $p$-path $i_0\dots i_p$ is denoted by $e_{i_0\dots i_p}$ and the elements of $\Lambda_p(V)$ are called \emph{paths}. The linear maps $\partial(e_{i_0\dots i_p})\coloneqq \sum_{q=0}^p (-1)^q e_{i_0\dots\widehat{i_q}\dots i_p}$ define the boundary operator $\partial\colon \Lambda_p\to\Lambda_{p-1}$, hence a chain complex \cite[Lemma~2.1]{Grigoryan_first}. 

\begin{definition}\cite[Definition~3.1]{Grigoryan_first}
	A \emph{path complex} over $V$ is a non-empty collection  $P$ of elementary paths on $V$ with the property that if $i_0\dots i_p$ belongs to $P$, then also $i_0\dots i_{p-1}$ and $i_1\dots i_p$ belong to $P$. The paths in $P$ are called \emph{allowed}.
\end{definition}

For a path complex $P$ on a set $V$, the vector space $\mathcal{A}_p(P)$ spanned by all the allowed $p$-paths from $P$ is a subspace of $\Lambda_p$. Define the subspace $\Omega_p(P)$ of $\mathcal{A}_p(P)$ as
\[
\Omega_p(P)\coloneqq \left\{v\in \mathcal{A}_p(P)\mid \partial v\in \mathcal{A}_{p-1}(P)\right\}.
\]
The elements of \(\Omega_p(P)\) are called the \emph{$\partial$-invariant} paths of $P$. Then the boundary operator $\partial$ restricts to a boundary operator on $\Omega_p(P)$~\cite[Section~3.2]{Grigoryan_first}, and provides a chain complex $(\Omega_n(P),\partial)$. 

\begin{definition}
The \emph{path homology} groups $\mathrm{PH}_n(P)$ of the path complex $P$ are the homology groups of the chain complex $(\Omega_n(P),\partial)$. 
\end{definition}

To every directed graph there is an associated path complex \cite[Ex.~3.3]{Grigoryan_first}. If $\GG=(V,E)$ is a digraph, an elementary $p$-path $i_0\dots i_p$ is allowed if $(i_{k-1},i_k)\in E$ for all $k=1,\dots,p$. The set of allowed $p$-paths on $\GG$ is denoted by $P_p(\GG)$; note that  $P_0(\GG)=V$, and $P_1(\GG)=E$. Then, the union $P(\GG)\coloneqq \bigcup_n P_n(\GG)$ is a path complex. In particular, if $\GG$ is a digraph, then the path homology of $\GG$ is defined as the path homology of the path complex $P(\GG)$, after restricting to the $\partial$-invariant paths. 

\begin{remark}
It has been shown that path homology satisfies nice functorial properties: for a morphism of digraphs  $f\colon \GG_1\to\GG_2$ one gets a homomorphism $f_*\colon\mathrm{PH}_*(\GG_1) \to \mathrm{PH}_*(\GG_1)$ between the associated path homology groups  \cite[Theorem~2.10]{grigoryan2014homotopy}. Furthermore, it has been shown that path homology satisfies K{\"u}nneth formulas \cite{grigor2017homologies} and analogous properties to the classical Eilenberg-Steenrod axioms \cite{Grigorianstax}.
\end{remark}

The following example illustrates a computation of the path homology groups of the same digraph as in Example~\ref{ex:square}.

\begin{example}\label{ex:sqfrp}
We return to the the square-shaped digraph $\GG=(V,E)$ on the vertices $0,1,2,3$ introduced in Example~\ref{ex:square}. The path complex $P(\GG)$ associated to $\GG$ is the complex with elementary allowed paths  as depicted below:
\begin{center}
	\begin{tikzpicture}
	\tikzstyle{arc}=[color=blue,shorten >= 3pt,shorten <= 3pt,->, thick]

	\node[] (0) at  (0,0) {0};
	\node[] (1) at  (0,1.5) {1};
	\node[] (2) at  (1.5,0) {2};
	\node[] (3) at  (1.5,1.5) {3};
	
	\draw[arc] (0) to (1);
	\draw[arc] (0) to (2);
	\draw[arc] (1) to (3);
	\draw[arc] (2) to (3);
	
	\node[right] at (3,1.5) {\(P_0=\{0,1,2,3\}=V\)};
	\node[right] at (3,0.75) {\({\color{blue}P_1=\{01, 02, 13, 23\}=E}\)};
	\node[right] at (3,0) {\({\color{magenta}P_2=\{013,023\}}\)};
	
	\draw[magenta,->] (-0.2,0.2) .. controls (-0.4,1.9) and (-0.4,1.9) .. (1.3,1.7);
	\draw[magenta,->] (0.2,-0.2) .. controls (1.9,-0.4) and (1.9,-0.4) .. (1.7,1.3);			
	\end{tikzpicture}
\end{center}
All the other $P_n$ are empty. As the edge $(0,3)$ does not belong to $\GG$, the paths $e_{013}$ and $e_{023}$ are not $\partial$-invariant. However, the linear combination $e_{013}-e_{023}$ would be, as \[\partial(e_{013}-e_{023})=e_{13}-e_{03}+e_{01}-e_{23}+e_{03}-e_{02}=e_{13}+e_{01}-e_{23}-e_{02}.\]
Therefore, if $\K$ is a field, then $\Omega_0=\langle e_0,e_1,e_2,e_3\rangle$, 
$\Omega_1=\langle e_{01}, e_{02}, e_{13}, e_{23} \rangle$ and $\Omega_2=\langle e_{013}-e_{023} \rangle$. The associated path homology of $\GG$ is then $\mathrm{PH}_0(P(\GG))\cong \K$ and it is $0$ in higher dimensions. This can be shown by direct computation from the chain complex $(\Omega_n(P),\partial)$, or by using the fact that the square digraph is contractible \cite[Example~3.13]{grigoryan2014homotopy}.
Observe that if we add the edge $(0,3)$ to $\GG$ the spaces $\Omega_1$ and $\Omega_2$ become $\Omega_1=\langle e_{01}, e_{02}, e_{13}, e_{23}, e_{03} \rangle$ and $\Omega_2=\langle e_{013},e_{023} \rangle$, but the path homology groups do not change. 
\end{example}

 Despite being a very interesting theory, path homology is not easily computable in higher degrees. As far as the authors know, there is no efficient algorithm for computing path homology groups in degree $i\geq 2$; for an algorithm in degree~$1$, see for instance~\cite{dey2020efficient}. 

\subsection{Hochschild homology of path algebras}
\label{HH}

Hochschild (co-)homology, introduced by Hochschild \cite{MR11076}, is a natural invariant of associative, not necessarily commutative, unital  algebras.  An investigation of this homology theory is beyond the purposes of this work, and we refer to \cite{loday} for a general and comprehensive overview on the subject. To each digraph $\GG$, we associate an  algebra~$R\GG$, called the \emph{path algebra} -- see Definition \ref{defpathalg}. The path algebra is associative and, for finite digraphs, also unital. Therefore, Hochschild homology groups of $R\GG$ give (algebraic) invariants of digraphs.  Goal of this section is to review some properties of Hochschild (co-)homology, and to introduce a new characteristic of digraphs (Definition~\ref{HHX}).

\subsubsection{Hochschild homology}
We first recall the definition of Hochschild homology, following \cite[Section~1.1]{loday}. For a commutative ring $R$, let  $A$ be an associative unital $R$-algebra; for example, $A$ can be a polynomial algebra over~$R$. For a bimodule\footnote{A bimodule over the algebra $A$ is an $R$-module endowed with an action of $A$ both on the left and on the right, such that $(am)a'=a(ma')$ for all $a, a'\in A$, and $m\in M$. The actions are compatible, and if $A$ is unital, the unit acts as the identity.} $M$ over $A$, let $C_n(A,M)$ be the $R$-module  $$C_n(A,M)\coloneqq M\otimes A^{\otimes n}$$ defined as the tensor product of \(M\) and $n$ copies of $A$ (all the tensor products being  over $R$). The boundary operator, classically denoted by $b$, is the $R$-linear map $b\colon  C_n(A,M)\to C_{n-1}(A,M)$ defined as follows:
\[
b(m,a_1,\dots,a_n)= (ma_1,a_2,\dots,a_n)+\sum_{i=1}^{n-1} (-1)^i(m,a_1,\dots, a_i a_{i+1},\dots,a_n)+ (-1)^n(a_n m,a_1,\dots, a_{n-1})
\]
In the formula, for simplicity of notation, we have dropped the tensor products.
The map $b$ is a boundary operator \cite[Lemma~1.1.2]{loday} and the pair $(C_*(A,M),b)$ is a chain complex, called the \emph{Hochschild complex}.

\begin{definition}\label{defhh}
	The \emph{Hochschild homology} groups $\mathrm{HH}_*(A,M)$ of an associative unital algebra $A$ with coefficients in a bimodule $M$ are the homology groups of the Hochschild complex.
\end{definition}

When $M$ is the bimodule $A$ itself, we use a simpler notation:
\begin{notation}
The {Hochschild homology} groups of an associative unital algebra $A$ with coefficients in $A$ are denoted by $\mathrm{HH}_*(A)$.
\end{notation}

 For illustration, we provide some elementary computations; for further details see \cite{loday}.

\begin{example}
	Let $M=A=R$ be a commutative ring. Then, we have
	\[
\mathrm{HH}_*(R)=
	\begin{cases}
		R & \text{ if } *=0 \\
		0 & \text{otherwise} 
	\end{cases}
	\ .
\]
In fact, as tensor products are computed over $R$ and $R\otimes_R R\cong R$, the chain complex $C_*(R,R)$ is a copy of $R$ in each degree, and the boundary operator $b$ is either the identity or the zero map, depending on the parity of $n$. The computation follows.
\end{example}

\begin{example}
	Let $M=A$ be an associative algebra. Then,
	$$\mathrm{HH}_0(A)=A/[A,A]$$
	where $[A,A]$ denotes the commutator submodule generated by all the $[a,a']=aa'-a'a$. In fact, for $a\otimes a'\in C_1(A,A)$, we have $b(a\otimes a')=aa'-a'a$. In particular, if $A$ is commutative, then we obtain $\mathrm{HH}_0(A)\cong A$, and for $A$ non-commutative  $\mathrm{HH}_0(A)$ coincides with its center.
\end{example}

 We recall that if $A$ is unital and commutative, then  the module of K\"ahler differentials~$\Omega^1(A)$ is the $A$-module generated by all the formal differentials $da$, with $a\in A$, which are $R$-linear, \emph{i.e.}, $d(\lambda a + \mu b)=\lambda da +\mu db$ for all $a,b\in A$ and $\lambda,\mu\in R$, and satisfy the additional product condition $d(ab)=a db+b da$ for all $a,b\in A$.
 Then the first Hochschild homology group of~$A$ is isomorphic to $\Omega^1(A)$ over the ring $R$ (\cite[Definition~1.1.9 \& Proposition~1.1.10]{loday}).
Analogously, in higher degrees, the $n$-th Hochschild homology~$\mathrm{HH}_n(A)$ of $A$ is related to the module of $n$-forms (\cite[Section~1.3]{loday}).

\begin{example}\label{ex:kX}
Let $A=\K[X]$ be the polynomial algebra over a field~$\K$. As $A$ is a commutative algebra, then $\mathrm{HH}_0(A)\cong A$. The first Hochschild homology group is isomorphic to the ideal~$(X)$ of $A$. In fact, by \cite[Proposition~1.1.10]{loday} it is isomorphic to the module of K\"ahler differentials~$\Omega^1(A)$, and this latter is generated by $X$. All its higher Hochschild homology groups are zero. 
\end{example}

\begin{remark}
    Note that the Hochschild homology groups $\mathrm{HH}_*(A,M)$ depend on the choice of the ground ring. Non-isomorphic ground rings can lead to different computations -- see for example \cite[Section~1.1.18]{loday}.
\end{remark}

A dual cohomological theory can be easily derived. In fact, the Hochschild cohomology groups~$\mathrm{HH}^*(A,M)$ of an associative algebra~$A$ with coefficients in $M$ can be defined as the homology groups of the cochain complex $\mathrm{Hom}_R(C_n(A,M))=\mathrm{Hom}_R(A^{\otimes n},M)$ -- \emph{cf.}~\cite{happel,loday}.

\begin{remark}\label{rmk:HH funct}
	The Hochschild homology construction is functorial in both the bimodule~$M$ and the algebra~$A$, and for $M=A$, Hochschild homology is a covariant functor from the category of associative $R$-algebras to the category of $R$-modules \cite[Section~1.1.4]{loday}. More precisely, a bimodule homomorphism $f\colon M\to M'$ induces a map
	\[
	f_*\colon \mathrm{HH}_*(A,M)\to \mathrm{HH}_*(A,M')
	\]
	 in Hochschild homology by sending the element~$(m,a_1,\dots, a_n)\in C_n(A,M)$ to the element $(f(m),a_1,\dots, a_n)\in C_n(A,M')$; the differential $b$ clearly commutes with it and induces the required homomorphism of homology groups. Similarly, if $M=A$ and $\varphi\colon A\to A'$ is an $R$-algebra homomorphism, then $\varphi$ extends to the tensor products giving a map $(a_0,\dots,a_n)\mapsto (g(a_0),\dots, g(a_n))$ and providing a morphism of chain complexes $C_*(A,A)\to C_*(A',A')$; hence, a homomorphism of Hochschild homology groups. 
\end{remark}

We note here that the functoriality with respect to $R$-algebra homomorphisms does not extend to Hochschild cohomology -- see, for instance, \cite[Section~1.5.5]{loday}.
This lack of functoriality  is the reason why we will focus on Hochschild homology, rather than cohomology.
	
\subsubsection{Hochschild homology of path algebras}\label{sec:HHpa}

In this subsection we apply Hochschild homology to certain algebras associated to digraphs. All results and proofs provided here are classical, and we claim no originality. For computational purposes, we also  restrict to the case in which the base ring \(R\) is an (algebraically closed) field~$\K$; we can think of $\K$ as the field of complex numbers.

Let $\GG=(V,E,s,t)$ be a (finite) directed graph. By a \emph{path} in $\GG$ we mean a sequence $\gamma=(e_1,\dots,e_n)$ of composable  
edges $e_i$ in $\GG$ such that $s(e_{i+1})=t(e_i)$. The number $n$ is the \emph{length} of the path. For any vertex $v$ of $\GG$, we also consider the  trivial path~$e_v$ of length $0$ at the vertex $v$.

\begin{definition}\label{defpathalg}
	The \emph{path algebra} $\K\GG$ associated to the digraph $\GG$ is the $\K$-vector space with a basis consisting of all possible paths in $\GG$, and the multiplication being defined 
	on two basis paths \(\gamma = (e_1,\dots,e_n)\), \(\gamma' = (e'_1,\dots,e'_p)\) by the formula
	\[\gamma \gamma' = \begin{cases}
		(e_1,\dots,e_n,e'_1,\dots,e'_p) &\text{ if } s(e'_1)=t(e_n)\\
		0 &\text{ otherwise} 
	\end{cases} \ .\]
\end{definition}

The following lemma is easily derived from the definition:

\begin{lemma}
The path algebra $\K\GG$ associated to a digraph $\GG$ is an associative algebra over $\K$, and has a unit if the digraph is finite.
\end{lemma}

\begin{proof}
	Let $\gamma, \gamma', \gamma''$ be paths in $\GG$. Then, by definition of product of paths, both $(\gamma\gamma')\gamma''$ and $\gamma(\gamma'\gamma'')$ are given as concatenation of paths if they are all compatible, and are both $0$ otherwise (by bilinearity). Since the algebra $\K\GG$ is defined as a vector space generated by the possible paths in $\GG$, this is enough to show its associativity. 
	
	If the digraph $\GG$ is finite, then the element $$e\coloneqq \sum_{v\in V(\GG)}e_v$$ is a unit. In fact, if $\gamma= (e_1,\dots,e_n)$ is a path of $\GG$, then $e_{s(e_1)}\gamma=\gamma$ and $\gamma e_{t(e_n)}=\gamma$, where $e_{s(e_1)}$ and $e_{t(e_n)}$ are the constant paths at $s(e_1)$ and $t(e_n)$, respectively. All the other products with paths of type $e_v$, for $v$ a vertex of $\GG$, are zero; hence we have $e\gamma= e_{s(e_1)}\gamma=\gamma$, and $\gamma e=\gamma e_{t{e_n}}=\gamma$.
\end{proof}

 The path algebra~$\K\GG$ is generated (as an algebra over $\K$) by all the paths of length at most~$1$, and the unit is the sum of all trivial paths. More precisely, it has the structure of a graded vector space with grading induced by the length of paths~\cite{Brion}. Observe that all trivial paths $e_v$, with $v\in V(\GG)$, are idempotents of $\K\GG$, as $(e_v)^2=e_v e_v=e_v$. Furthermore, if $v$ and $w$ are distinct vertices of~$\GG$, then $e_v e_w=0$.
 
 We will be interested in digraphs without oriented cycles. In fact, we have the following:
 
 \begin{proposition}\label{prop:HH finite}
 	The path algebra $\K\GG$ is of finite dimension\footnote{here as a $\K$-vector space} if, and only if, $\GG$ is finite and has no oriented cycles. 
 \end{proposition}

\begin{proof}
	The path algebra $\K\GG$ is generated by the paths of length at most one, this number being bounded by the number of all vertices and edges of $\GG$. 
	
	To prove the statement, it is enough to show that the number of paths of $\GG$ is finite if, and only if, $\GG$ is finite and has no oriented cycles. Assume first that $\GG$ is finite without oriented cycles. Then the number of paths of length $l$ in $\GG$ is bounded by $|E(\GG)|^l$. If the number of paths of $\GG$ is infinite, then  there is a path of arbitrary large length. In particular, there exists a path~$\gamma$ with length greater than the number of vertices of $\GG$. This leads to a contradiction, because there exists a vertex $v$ of $\GG$ encountered twice by $\gamma$, hence an oriented cycle. 
	Conversely, any infinite graph, or any oriented cycle of $\GG$, gives infinitely many paths.
\end{proof}

	\begin{figure}[h]
		\centering
		\begin{tikzpicture}[baseline=(current bounding box.center)]
			\tikzstyle{point}=[circle,thick,draw=black,fill=black,inner sep=0pt,minimum width=2pt,minimum height=2pt]
			\tikzstyle{arc}=[shorten >= 8pt,shorten <= 8pt,->, thick]
			
			\node[above] (v0) at (0,0) {$v_1$};
			\draw[fill] (0,0)  circle (.05);
			\node[above] (v1) at (1.5,0) {$v_2$};
			\draw[fill] (1.5,0)  circle (.05);
			\node[] at (3,0) {\dots};
			\node[above] (v4) at (4.5,0) {$v_{n-1}$};
			\draw[fill] (4.5,0)  circle (.05);
			\node[above] (v5) at (6,0) {$v_{n}$};
			\draw[fill] (6,0)  circle (.05);
			
			\draw[thick, red, -latex] (0.15,0) -- (1.35,0);
			\draw[thick, red, -latex] (1.65,0) -- (2.5,0);
			\draw[thick, red, -latex] (3.4,0) -- (4.35,0);
			\draw[thick, red, -latex] (4.65,0) -- (5.85,0);
		\end{tikzpicture}
		\caption{The linear $n$-graph $I_n$.}
		\label{fig:nstep}
	\end{figure}
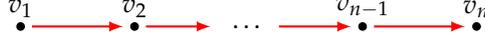
	
	We proceed with some elementary examples of path algebras:

\begin{example}
	The simplest graphs to consider are given by the graph $\GG_0$ with a single vertex~$v$, the directed graph~$\GG_1$ with a vertex $v$ and a loop $e$ at $v$, and the digraph $\GG_2$ with two vertices~$v_0,v_1$ and a directed edge $e_1=(v_0,v_1)$ between them. The associated path algebras are given by the base field $\K$, the polynomial ring $\K[X]$, and by the ring of upper triangular $(2\times 2)$-matrices. More generally, if $I_n$ is the graph illustrated in Figure~\ref{fig:nstep}, then the path algebra~$\K I_n$ is isomorphic to the 
	ring of upper triangular $(n\times n)$-matrices.
	An isomorphism can be described by sending the trivial path $e_{v_i}$  to the entry $(i,i)$, and edge $(v_i,v_{i+1})$ to the entry~$(i,i+1)$.
\end{example}

Observe that the map assigning to a digraph  its path algebra is functorial:

\begin{observation}\label{obsHHfunctor}
The assignment $\GG\mapsto \K\GG$ associating  the path algebra $\K\GG$ to a digraph~$\GG$ extends to a functor from the category of directed graphs to the category \(\K\text{-}\mathbf{Alg}\) of associative $\K$-algebras. To see this, we first observe that a morphism of digraphs sends distinct vertices to distinct vertices and, for  consecutive edges, it preserves sources and targets. Hence, a morphism of digraphs sends paths to paths, induces a $\K$-linear map between the vector spaces, and preserves the composition of paths. As a consequence, the induced map between the associated path algebras uniquely extends by linearity to a $\K$-algebra homomorphism.
Hochschild homology of associative path algebras is functorial -- see Remark~\ref{rmk:HH funct}. Then, the composition
\begin{equation}\label{eq:HH}
\mathbf{Digraph}\xrightarrow{\K-} \K\text{-}\mathbf{Alg}\xrightarrow{\mathrm{HH}_n}\mathbf{Vect}
\end{equation}
describes the Hochschild homology of the path algebra of a (finite) digraph as a functor on the category $\mathbf{Digraph}$ with values in the category~$\mathbf{Vect}$ of vector spaces over $\K$. 
\end{observation}

Computations of Hochschild (co-)homology groups may be difficult for arbitrary associative algebras, but when $A$ is the path algebra~$\K\GG$ of a directed graph, computations are easier and reflect the combinatorial properties of the digraph~$\GG$. First, it is a standard fact that the path algebra associated to a digraph is a hereditary algebra, \emph{i.e.}, all submodules of its projective modules are projective. In fact, modules over a path algebra have a standard resolution of length~$1$ \cite[Proposition~1.4.1]{Brion}, and as a consequence the Hochschild homology groups  $\mathrm{HH}_*(A)$ of the path algebra $A$ vanish in degree $\geq 2$.  
In degrees $0$ and $1$, the computation is due to Happel~\cite{happel} 
(see also \cite[Proposition~4.4]{redondo}):

\begin{theorem}\label{thmhapp}
	If $\GG=(V(\GG),E(\GG),s,t)$ is a connected  directed graph without oriented cycles and $\K$ is an algebraically closed field, then 
	\[
	\dim_{\K} \mathrm{HH}^i(A)=\dim_{\K} \mathrm{HH}_i(A)=
	\begin{cases}
		1 \quad &\text{ if } i=0 \\
		0 \quad &\text{ if } i>1 \\
		1-n +\sum_{e\in E(\GG)}\dim_{\K} e_{t(e)}A e_{s(e)}  \quad &\text{ if } i=1
	\end{cases}
	\] 
	where $A=\K\GG$ is the path algebra of $\GG$,  $n=|V(\GG)|$ is the number of vertices of $\GG$ and $e_{t(e)}A e_{s(e)}$ is the subspace of $A$ generated by all the possible paths from $s(e)$ to $t(e)$ in $\GG$.
\end{theorem}  

In general, for infinite digraphs, or digraphs admitting cycles, this computation can not be used and the first Hochschild homology group is of infinite rank.

\begin{example}
Let $\GG$ be the digraph with a vertex $v$ and the single directed edge (a loop) $(v,v)$. The path algebra $\K\GG$ is isomorphic to the polynomial algebra $\K[X]$. This is a commutative algebra over $\K$, hence by  Example~\ref{ex:kX}, we get
	\[
	\mathrm{HH}_i(\K\GG)\cong
	\begin{cases}
		\K[X] \quad &\text{ if } i=0 \\
		(X) \quad &\text{ if } i=1 \\
		0 \quad &\text{ if } i>1,
	\end{cases}
	\] 
which is not finite over $\K$.
\end{example}

In concrete applications, one avoids the case of digraphs with paths  of infinite length by restricting to some special classes of acyclic directed subgraphs. 
In such restricted context, we observe that Hochschild homology can be seen as a functor with values in the category of  (graded)  vector spaces of {finite} dimension: 

\begin{observation}\label{obs:HHfunctorfin}
	Theorem~\ref{thmhapp} implies that, when restricting to digraphs without oriented cycles, the associated Hochschild (co-)homology groups are vector spaces of finite dimension. Let $\mathbf{Digraph}_0$ be the subcategory of $\mathbf{Digraph}$ consisting of finite digraphs without oriented cycles and induced morphisms of digraphs. Then, the composition in Eq.~\eqref{eq:HH} induces the composition of functors
		\[
	\mathbf{Digraph}_0\xrightarrow{\K-} \K\text{-}\mathbf{Alg}\xrightarrow{\mathrm{HH}_n}\mathbf{FinVect}\]
	where now the target category $\mathbf{FinVect}$ is the category of \emph{finite-dimensional} vector spaces over~$\K$.
\end{observation}

We give a concrete example:

\begin{example}
	Let $\phi\colon \GG_1\to \GG_2$ be the regular morphism of digraphs illustrated in Figure~\ref{fig:diag square morph}, and defined by sending $v_i$ to the vertex $w_i$, for $i=0,1,2$. 
	The morphism~$\phi$ sends paths of $\GG_1$ to paths of $\GG_2$ of the same length, thus inducing by  $\K$-linearity an homomorphism of vector spaces $\phi_*\colon \K\GG_1\to \K\GG_2$. In terms of basis elements, the trivial path~$e_{v_i}$ is sent to the trivial path~$e_{w_i}$, whereas the basis elements corresponding to the $1$-paths $(v_0,v_1)$, $(v_0,v_2)$, and $(v_1,v_2)$ are sent to those of~$\K\GG_2$ corresponding to $(w_0,w_1)$, $(w_0,w_2)$, and $(w_1,w_2)$. The tensor product operations are clearly compatible. As in Remark~\ref{obs:HHfunctorfin}, we get morphisms 
	\[
	\phi_n\colon C_n(\K\GG_1,\K\GG_1)= \K\GG_1\otimes\dots\otimes \K\GG_1\longrightarrow \K\GG_2\otimes\dots\otimes \K\GG_2 =C_n(\K\GG_2,\K\GG_2),
	\]
	hence induced maps between the Hochschild (co-)chain complexes. By applying the functors $\mathrm{Hom}(-,\K\GG_{-})$ to the minimal projective resolutions of $\K\GG_1$ and $\K\GG_2$, we get a diagram of short exact sequences (for such a computation, see \cite{happel})
	\[
	\begin{tikzcd} 
	0	\arrow[r] & \K \arrow[d,"\mathrm{Id}_{\K}"']\arrow[r]& \K^3\cong \K^{|V(\GG_1)|}\arrow[r]\arrow[d,"\phi_*"'] & \K^4 \arrow[d,"\phi_*"'] \arrow[r] &0 \\ 
		0	\arrow[r] & \K \arrow[r]& \K^4\cong \K^{|V(\GG_2)|}\arrow[r] & \K^8  \arrow[r] &0 \\ 
	\end{tikzcd}\]
	where the maps are induced by identification of paths of length $0$ (central map) and of the paths $(v_0,v_1)$, $(v_0,v_2)$, $(v_1,v_2)$, and $(v_0,v_1,v_2)$ with the respective ones in $\GG_2$ (rightmost map).
	\begin{figure}
		\centering
	\begin{tikzpicture}[scale=0.8][baseline=(current bounding box.center)]
		\tikzstyle{point}=[circle,thick,draw=black,fill=black,inner sep=0pt,minimum width=2pt,minimum height=2pt]
		\tikzstyle{arc}=[shorten >= 8pt,shorten <= 8pt,->, thick]		
		
		\node  (v0) at (0,0) {};
		\node[below] at (0,0) {$v_0$};
		\draw[fill] (0,0)  circle (.05);

		\node  (v2) at (1.5,2) { };
		\node[above] at (1.5,2) {$v_{1}$};
		\draw[fill] (1.5,2)  circle (.05);
		\node  (v3) at (3,0) { };
		\node[below]  at (3,0) {$v_{2}$};
		\draw[fill] (3,0)  circle (.05);
		
		\draw[thick, red, -latex] (v0) -- (v2);
		\draw[thick, red, -latex] (v2) -- (v3);
		\draw[thick, red, -latex] (v0) -- (v3);
		
		\draw[thick,  -latex] (4,1) -- (7,1);
	\end{tikzpicture}
	\begin{tikzpicture}[scale=0.7][baseline=(current bounding box.center)]
		\tikzstyle{point}=[circle,thick,draw=black,fill=black,inner sep=0pt,minimum width=2pt,minimum height=2pt]
		\tikzstyle{arc}=[shorten >= 8pt,shorten <= 8pt,->, thick]		
		
		\node (o) at (-1.5,0) {};
		
		\node  (v0) at (0,0) {};
		\node[below] at (0,0) {$w_0$};
		\draw[fill] (0,0)  circle (.05);
		\node  (v1) at (0,3) { };
		\node[above]  at (0,3) {$w_3$};
		\draw[fill] (0,3)  circle (.05);
		\node  (v2) at (3,3) { };
		\node[above] at (3,3) {$w_{1}$};
		\draw[fill] (3,3)  circle (.05);
		\node  (v3) at (3,0) { };
		\node[below]  at (3,0) {$w_{2}$};
		\draw[fill] (3,0)  circle (.05);
		
		\draw[thick, red, -latex] (v0) -- (v1);
		\draw[thick, red, -latex] (v0) -- (v2);
		\draw[thick, red, -latex] (v2) -- (v3);
		\draw[thick, red, -latex] (v0) -- (v3);
		\draw[thick, red, -latex] (v1) -- (v2);
	\end{tikzpicture}
	\caption{Regular morphism of digraphs $\phi(v_i)=w_i$.}
	\label{fig:diag square morph}
\end{figure}
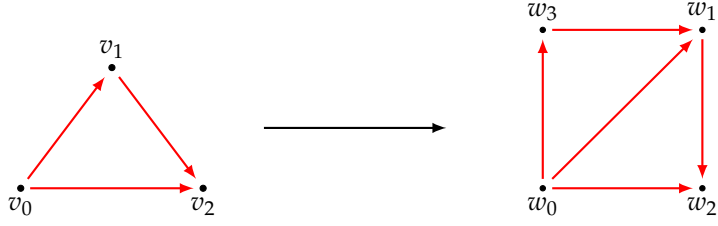
Focusing on the first Hochschild homology groups $\mathrm{HH}_1(\K\GG_1)\cong \K^2$ and $ \mathrm{HH}_1(\K\GG_2)\cong \K^5$,  this roughly describes the functoriality of Hochschild homology by means of the paths in the digraphs; the first Hochschild homology groups being obtained as alternating sums of the vector spaces appearing in the horizontal diagrams of the short exact sequences.
\end{example}

In general, descriptions of the Hochschild homology groups are not easy. In this work, we use the computation of Theorem~\ref{thmhapp} as handleable. In order not to loose the information captured by the number of cycles, we will also consider the following characteristic measure:

\begin{definition}\label{HHX}
	Let $\GG=(V(\GG),E(\GG),s,t)$ be a digraph,  and $A=\K\GG$ the path algebra of $\GG$ over a field~$k$. The \emph{Hochschild characteristic} of $\GG$ is defined as
	\[
	\mathcal{X}_{\mathrm{HH}}(\GG)\coloneqq \dim_k \mathrm{HH}_0(A)-\left(1-n +\sum_{e\in E(\GG)}\dim_k e_{t(e)}A e_{s(e)}\right)+C(\GG)
	\]
	where $n=|V(\GG)|$ and $C(\GG)$ is the number of simple oriented cycles in $\GG$, i.e.\ cycles with only the first and last vertices being equal.
\end{definition}

Note that the Hochschild characteristic agrees with the Euler characteristic of the Hochschild chain complex associated to the path algebra $k\GG$, if $\GG$ has no oriented cycles and $k$ is algebraically closed. The definition of {Hochschild characteristic} extends then to any field.
We conclude with an illustrative  example.
\begin{example}
	We apply the Hochschild homology computation of Theorem~\ref{thmhapp} to the square digraph~$\GG$ of Example~\ref{ex:square} and Example~\ref{ex:sqfrp}. The Hochschild homology group in degree~$0$ is isomorphic to $\K$, whereas the dimension of $ \mathrm{HH}_1(\K\GG)$ is $1$. The computation in degree $1$ changes if we add to $\GG$ the edge $(0,3)$, giving $\dim_{\K} \mathrm{HH}_1(\K\GG)=4$. 
\end{example}

\subsection{A view towards other approaches}\label{sechomothers}

We conclude this survey section by reviewing some generalizations and other approaches to (co-)homology theories of digraphs; the literature on the subject is very rich, and this survey is far from being exhaustive.
\begin{itemize}
    \item A generalization of the homology of  directed graphs as the homology of the associated flag complex (see Section~\ref{secflcompl}) is given by the homology of the so-called \emph{flag tournaplex}~\cite{govc2021complexes}. 
In fact, when a directed graph has no reciprocal edges the associated flag tournaplex
is isomorphic to the flag complex of the underlying undirected graph. The flag tournaplex of a digraph has been employed  as a classifier in \cite{govc2021complexes}, combined with  directionality invariants and persistent homology methods.
\item The theory of path homology  for digraphs, as developed in \cite{Grigoryan_first}, has been further extended to ring coefficients \cite{li2020geometric} and to multigraphs and quivers \cite{Grigoryan} or to more general  path algebras, \emph{e.g.}, to the realm of differential algebras \cite{ren2021differential}. It has also a cohomological counterpart \cite{Grigoriancoho}. This has also been extended to a persistent path homology approach -- see for instance \cite{dey2020efficient}.
\item Hochschild homology can be endowed with an additional differential of degree $1$, usually denoted by $B$, turning it into a \emph{mixed complex} \cite{KASSEL1987195}. The additional differential $B$ leads to the construction of the so-called \emph{cyclic homology}, and to its variations \emph{negative cyclic homology} and \emph{periodic cyclic homology} of algebras -- see also \cite{loday}. Therefore, application of such homology theories to the path algebra of a digraph may lead to other invariants, extending the approach surveyed in Section~\ref{sec:HHpa}.
\item Ordinary homology groups of digraphs, and the Hochschild homology groups of the associated path algebras, have  natural generalizations to the categorical framework.
 One way to do so is by replacing the path algebra of Definition~\ref{defpathalg} with a suitable (freely generated) category~$\mathbf{Path}(\GG)$, called the \emph{path category} -- see also the discussion below. 
 In a similar fashion, instead of constructing the path algebra or the path category, one can associate to a directed graph other mathematical objects, for example the so-called   path posets~$P(\GG)$. First introduced by Turner and Wagner~\cite{turner}, the path poset has been recently used to define new combinatorial cohomologies of digraphs \cite{secondo, caputi2021multipath}, and that can be generalized to arbitrary \emph{monotone} properties of graphs~\cite{monotone}. This approach seems to be related to other topological/combinatorial invariants of simplicial complexes~\cite{caputi2022categorifying}. 
 \item Among other approaches,  sheaf homology
  has been used with some applications to the  Max-Flow Min-Cut  theorem \cite{krishnan2014flowcut}, as well as a directed approach to algebraic topology built upon  cohomology of small categories with coefficients in natural systems as in \cite{BAUES1985187} and \cite{10.1007/978-3-662-47666-6_14} and in the references therein.
\end{itemize}

We can illustrate these various approaches in the following (non-commutative) diagram
\[
\begin{tikzcd}
 & &\mathbf{Cat} \arrow[dddrr,bend left,  "H"]& &  \\
 & & \mathbf{Poset}\arrow[u, ]\arrow[ddrr,  "H"] && \\
 & &\mathbf{SCpx}\arrow[u, ]\arrow[drr, "H"]&& \\
\mathbf{DiGraph} \arrow[drr, "\K-"] \arrow[rrrr, dashed, "\text{Homologies of digraphs}"]\arrow[urr, "dFl"]\arrow[uurr, "P"] \arrow[uuurr, bend left, "\mathrm{Free}"] & & & & \mathbf{Ab} \\
 & &\K\text{-}\mathbf{Alg}\arrow[urr, "\mathrm{HH}"]&& \\
\end{tikzcd}
\]
where we synthetically show the construction of homology theories of digraphs as representations of the category $\mathbf{DiGraph} $. Note that path homology does not appear in the picture because it is not yet known if it factors through simplicial complexes/algebras/posets/categories.

We wish to spend a few more words on the path category and related homology theories, as these are close to the Hochschild homology of path algebras.
The path category $\mathbf{Path}(\GG)$ associated to a directed graph $\GG$ is the category freely generated by the paths of $\GG$.
 There is a forgetful functor from the category $\mathbf{Cat}$ of small categories and functors to the category of quivers (thought of as directed graphs with loops and multiple edges). Such forgetful functor has a left adjoint,  the functor sending a quiver to the free category on that quiver (see \emph{e.g.},~\cite[Section~II.7]{maclane:71}). As each directed graph  is, in particular, also a quiver, one gets a functor 
\begin{equation}\label{eqfreecat}
\mathrm{Free}\colon\mathbf{DiGraph}\to\mathbf{Cat}
\end{equation}
from the category of digraphs to the category of small categories.
It is easy to see that $\mathrm{Free}(\GG)$ is in fact the path category $\mathbf{Path}(\GG)$. Then, the functor in \eqref{eqfreecat} allows one to use homology theories of categories for obtaining new  invariants of digraphs. The naive idea of directly computing the homology groups of the category~$\mathbf{Path}(\GG)$ (with constant coefficients), 
would not give new information due to the following (see \cite[Ex.~4.3]{citterio}):

\begin{proposition}\label{prop:hogr=homcat}
	The classifying space $|{N}(\mathbf{Path}(\GG))|$ of the path category  of a directed graph $\GG$ has the homotopy type of the geometric realization $|\GG|$ of the digraph $\GG$.
\end{proposition}

However, more interesting homology theories  arise when considering the homology of categories with coefficients in  functors \cite[App.~2]{Gabriel1967CalculusOF}. In this framework, one considers the homology groups of $\mathbf{Path}(\GG)$ with twisted coefficients in the same spirit as in usual homology of topological spaces but with local coefficients (see also \cite[Section~1]{quillenI}).
Remarkably, this point of view has been used in node embedding and community detection problems~\cite{kaul}.

\section{Connectivity structures and digraphs}\label{sec:connectivity_homology}

As outlined in Section \ref{homologies}, homology theories of digraphs are flexible and versatile tools, able to capture various topological information from the input graphs. However, as we will show with examples in Section~\ref{sec:whyt}, these homology theories might miss some combinatorial information inherent in the digraph structure. Therefore, with the aim of capturing both the topology and combinatorics of digraphs, in this section we will study some special cases of \emph{connectivity digraphs}, \emph{i.e.},~digraphs constructed by using some combinatorial information of $\GG$. For example, connectivity digraphs can be described by edges, paths, sets of edges, or cliques, together with their incidence relations. 

We will focus on two closely related \emph{simplicial} connectivity approaches and show that they capture the combinatorial information described by edges and ordered cliques in a digraph.
The first one, developed by the second author in \cite{Riihimaki_simplicial_connectivities} and briefly recalled in Section~\ref{sec:q_analysis}, generalizes the classical \emph{\(q\)-connectivity} analysis \cite{Atkin1972,Atkin1974_book} to the context of ordered simplicial complexes. This construction is based on ordered cliques sharing \(q\)-faces respecting a chosen directionality condition. The connectivity digraphs constructed have the additional structure of a preordered set (Definition \ref{def:q_near_directed}).
Then in Section~\ref{secnpathgr}, we investigate the particular case of connectivity digraphs built with ordered cliques and codimension~$1$ incidence relations. The connectivity digraph structure is induced from a total order of simplicial face maps. This choice is shown to generalize the notion of line digraphs  (going from the combinatorics of edges to the combinatorics of higher simplices), giving what we call here the \emph{$n$-path digraphs} -- cf.~Definition~\ref{npdigraph}.

\subsection{Connectivity digraphs}\label{sec:conn graphs}

Goal of this subsection is to introduce the concept of connectivity digraphs. For a digraph $\GG$, a connectivity digraph associated to $\GG$ is meant to encapsulate some combinatorial information of $\GG$. In order to make it formal, we start with the notion of a connectivity structure. 
\begin{definition}\label{def:conn_struct}
	A \emph{connectivity structure} is a triple $(\GG,\mathfrak{F}, {A})$ consisting of a digraph $\GG$, and a non-empty family $\mathfrak{F}$ of subgraphs of $\GG$ together with a $\{0,1\}$-valued function $A\colon \mathfrak{F}\times \mathfrak{F}\to \{0,1\}$.
\end{definition}

Note that there are no additional requirements on the map $A$; roughly, a connectivity structure is a way to encode the connectivity properties of families of subsets of $\GG$, all at once. 
\begin{remark}
    For a given connectivity structure $(\GG,\mathfrak{F}, {A})$ there is an associated well-defined digraph $E_{\mathfrak{F}}\GG$ (possibly with self-loops) constructed as follows: the set of vertices of $E_{\mathfrak{F}}\GG$ is the family $\mathfrak{F}$, and for $H_1$, $H_2$ in $\mathfrak{F}$, there is a directed 
edge $(H_1,H_2)$ in $E_{\mathfrak{F}}\GG$ if, and only if, 
$A(H_1,H_2)=1$.
\end{remark}
We can now give the formal definition of connectivity digraphs.
\begin{definition}\label{def:conn digraphs}
	A \emph{connectivity digraph} is the directed graph $E_{\mathfrak{F}}\GG$ associated to a connectivity structure $(\GG,\mathfrak{F}, {A})$.

	A \emph{morphism}  of connectivity structures $(\GG,\mathfrak{F}, {A})\to (\GG',\mathfrak{F}', {A}')$ is a morphism of digraphs $\Phi\colon \GG\to\GG'$ inducing a function $\phi\colon\mathfrak{F}\to \mathfrak{F}'$ such that $A(H_1,H_2)=1$ implies $A'(\phi(H_1),\phi(H_2))=1$.
\end{definition}

Connectivity structures and morphisms of connectivity structures form a category  where compositions are induced by compositions of morphisms of digraphs. It is also easy to see that the map which associates to a connectivity structure~$(\GG,\mathfrak{F}, {A})$ the connectivity digraph~$E_{\mathfrak{F}}\GG$ is functorial with respect to morphisms of connectivity structures. 

The context of connectivity structures is quite general. Different connectivity structures can encode various combinatorial information of digraphs. Our main motivation for introducing connectivity structures is that the associated connectivity digraphs provide domains for the homology theories described in Section~\ref{homologies}, and then one can extend the class of digraph invariants.

We remark here that a definition similar to Definition~\ref{def:conn digraphs} has previously appeared in \cite[Section~6]{Grigoryan} in the form of a connectivity multigraph of a CW-complex, and meant to extend Atkin's connectivity graphs (see Section \ref{sec:q_analysis}): the vertices of the multigraph are the $n$-cells and two vertices are adjacent if the corresponding cells share a face. This graph is given the structure of a directed graph by first numbering all the cells, and then using this enumeration for describing the directions of the edges. 
Our construction of connectivity digraphs, and our two main examples of simplicial connectivities provided in Section~\ref{sec:q_analysis} and Section~\ref{secnpathgr}, do not require a predetermined enumeration of the vertices, making the definition of connectivity digraphs more intrinsic. We also give a more immediate recipe for connectivity digraphs; in fact, Definition~\ref{def:conn_struct} essentially provides the adjacency matrices.

\subsection{{Motivating examples}}\label{sec:whyt}

We start by motivating our combinatorial constructions, provided in Section~\ref{sec:q_analysis} and Section~\ref{secnpathgr}, with some examples. Specifically we construct  non-isomorphic digraphs, such that  all the homology theories introduced in Section~\ref{homologies} fail to distinguish them. Then, we will see that the same homology theories, applied to the connectivity digraphs, are able to distinguish them -- see Example~\ref{exHHdigraps} and Example~\ref{exspheresHH}, proving that encoding the combinatorics is quite helpful. 

\begin{example}\label{expaths}
	Consider the following directed graphs:
	\begin{center}
		\[\GG_1=
		\begin{tikzpicture}[baseline=(current bounding box.center)]
		\tikzstyle{arc}=[shorten >= 3pt,red, shorten <= 3pt,->, thick]
		
		\node[] (0) at  (-2,-1) {0};
		\node[] (1) at  (-2,1) {1};
		\node[] (2) at  (0,-1) {2};
		\node[] (3) at  (2,1) {$3$};
		\node[] (4) at (2,-1) {4};
		\draw[arc] (0) to (1);
		\draw[arc] (0) to (2);
		\draw[arc] (2) to (4);
		\draw[arc] (1) to (2);
		\draw[arc] (3) to (1);
		\draw[arc] (3) to (2);
		\draw[arc] (3) to (4);
		\end{tikzpicture}
		\quad \text{ and } \quad \GG_2=
		\begin{tikzpicture}[baseline=(current bounding box.center)]
		
		\tikzstyle{arc}=[shorten >= 3pt,red, shorten <= 3pt,->, thick]			
		\node[] (0) at  (-2,-1) {0};
		\node[] (1) at  (-2,1) {1};
		\node[] (2) at  (0,-1) {2};
		\node[] (3) at  (2,1) {$3$};
		\node[] (4) at (2,-1) {4};
		\draw[arc] (0) to (1);
		\draw[arc] (0) to (2);
		\draw[arc] (2) to (4);
		\draw[arc] (1) to (2);
		\draw[arc] (1) to (3);
		\draw[arc] (2) to (3);
		\draw[arc] (4) to (3);
		\end{tikzpicture}
		\]
	\end{center}
	The two digraphs are not isomorphic. For example, the total degree of the vertex $2$  in $\GG_1$ is $4$ with out-degree $1$, but there are no vertices of out-degree $1$ and total degree $4$ in $\GG_2$. 
	First, observe that the directed flag complexes associated to  $\GG_1$ and $\GG_2$ are  both contractible. Hence, the associated simplicial homology groups are isomorphic.

	Following Section~\ref{secpathhom}, we find the $\partial$-invariant paths for \(\GG_1\) and \(\GG_2\): $\Omega_2(\GG_1)=\{e_{012}, e_{312}, e_{324}\}$ and $\Omega_3(\GG_1)=\emptyset$, and $\Omega_2(\GG_2)=\{e_{012}, e_{123}, e_{243}, e_{013}-e_{023}\}$ and $\Omega_3(\GG_2)=\{e_{0123}\}$, and all the other $\Omega_n$, with $n\geq 3$, are empty. The \(\Omega_1\) and \(\Omega_0\) are always spanned by the edges and vertices, respectively. Therefore, we get the chain complexes:
	\[
	(\Omega_*(\GG_1),\partial)\coloneqq 0\to \Omega_2(\GG_1)\to\Omega_1(\GG_1)\to\Omega_0(\GG_1),
	\]
	\[
	(\Omega_*(\GG_2),\partial)\coloneqq 0\to\Omega_3(\GG_2)\to \Omega_2(\GG_2)\to\Omega_1(\GG_2)\to\Omega_0(\GG_2).
	\]
	The homology groups of these chain complexes are both trivial and concentrated in degree $0$, with the only non-trivial path homology groups $\mathrm{PH}_0(P(\GG_1))\cong \K\cong \mathrm{PH}_0(P(\GG_2))$.
	
	In the case of the Hochschild homology of the path algebras, as both graphs are acyclic, we can use Theorem~\ref{thmhapp}. This gives us isomorphic \(\K\)-vector spaces in all degrees (with both $\dim_{\K} \mathrm{HH}_1(\K\GG_1)$ and $\dim_{\K} \mathrm{HH}_1(\K\GG_2)$ equal to $7$).
\end{example}

Remarkably, the same homology theories can distinguish the associated line digraphs --  the line digraphs having $2$ and $1$ connected components, respectively -- suggesting that  the directed combinatorics plays an important role. We give another example:

\begin{example}\label{exspheres}
	Consider the following digraphs with reciprocal edges:
	\begin{center}
		\begin{tikzpicture}
			\tikzset{vertex/.style = {minimum size=1.5em}}
			\tikzstyle{arc}=[shorten >= 1pt,red, shorten <= 1pt,->,thick]
			
			\node[left] at (-1.2,0) {\(\mathcal{S}_1=\)};
			\node[vertex] (0) at  (0,1.3) {0};
			\node[vertex] (1) at  (-1,0) {1};
			\node[vertex] (2) at  (1,0) {2};
			\node[vertex] (3) at  (0,-1.3) {3};
			\draw[arc] (0) to (1);
			\draw[arc] (0) to (2);
			\draw[arc] (2) to[bend left] (1);
			\draw[arc] (1) to[bend left] (2);
			\draw[arc] (2) to (3);
			\draw[arc] (1) to (3);
			
			\node[left] at (3.8,0) {\(\mathcal{S}_2=\)};
			
			\node[vertex] (0) at  (5,1.3) {0};
			\node[vertex] (1) at  (4,0) {1};
			\node[vertex] (2) at  (6,0) {2};
			\node[vertex] (3) at  (5,-1.3) {3};
			\draw[arc] (0) to (1);
			\draw[arc] (0) to (2);
			\draw[arc] (2) to[bend left] (1);
			\draw[arc] (1) to[bend left] (2);
			\draw[arc] (3) to (1);
			\draw[arc] (3) to (2);
		\end{tikzpicture}
	\end{center}
 The associated directed flag complexes in both cases are topologically $2$-spheres, hence their homology groups are isomorphic. A computation similar to the one  in the previous example shows that the associated path homologies are also isomorphic. The Hochschild homology groups of the path algebras associated to the graphs $\mathcal{S}_1$ and $\mathcal{S}_2$ are also isomorphic (in degree $1$ of infinite dimension over $\K$, both digraphs having oriented cycle). 
	\end{example}
 
Motivated by these examples, we proceed with investigating two examples of (simplicial) connectivity digraphs.

\subsection[q-connectivity]{\(q\)-connectivity}\label{sec:q_analysis}
Our first connectivity structure is an extension of Atkin's \(Q\)-analysis \cite{Atkin1972} as developed in \cite{Riihimaki_simplicial_connectivities}. We briefly summarise this theory as needed for the purposes of this paper, and guide the reader to the previous references for more in-depth expositions. The essential idea in Atkin's work is the generalisation of edge path connectivity of a simplicial complex to sequences of connected simplices through sharing of faces of certain dimension. 

\begin{definition}\leavevmode\label{def:q_connectivity}
Let \(K\) be a simplicial complex. Two simplices \(\sigma\) and \(\tau\) of \(K\) are \emph{\(q\)-near}, if they share a \(q\)-face. Two simplices \(\sigma\) and \(\tau\) of \(K\) are \emph{\(q\)-connected}, if there is a sequence of simplices in \(K\),
\[\sigma=\alpha_0,\alpha_1,\alpha_2,\dots,\alpha_n,\alpha_{n+1}=\tau,\]
such that any two consecutive ones are \(q\)-near. The sequence of simplices is called a \emph{\(q\)-connection} between \(\sigma\) and \(\tau\). 
\end{definition}
Similarly to the property of being path connected, we say that the complex \(K\) is \(q\)-connected if any two simplices in \(K\) of dimension greater than or equal to \(q\) are \(q\)-connected. The notion of \(q\)-connectivity is an equivalence relation on the set \(K_q\) of simplices of dimension \(q\) and higher, for \(0 \le q \le \dim(K)\). 

The aim of \(Q\)-analysis is to associate a simplicial complex with its \(q\)-connectivity equivalence classes, or its \(q\)-connected components. Note that if a simplex \(\sigma\) is maximal in $K$ with respect to inclusion and \(\text{dim}(\sigma) = q\), then \(\sigma\) is \(q\)-connected only to itself; hence every maximal \(q\)-simplex generates its own equivalence class. For each \(q\) the equivalence classes encode the connectivity information of the \(q\)-upper skeleton of \(K\). A related notion was introduced in \cite{clique_communities} to study community structures in networks.
\begin{definition}\label{def:conncliques}
	Let \(\GG\) be a graph and \(n \geq 2\) a natural number. Two \(n\)-cliques in \(\GG\) are connected if there is a sequence of \(n\)-cliques of \(\GG\) such that any two consecutive cliques share \(n-1\) vertices. A \emph{\(n\)-clique community} of \(\GG\) is a maximal set of pairwise connected \(n\)-cliques.
\end{definition}

The \(n\)-cliques of a graph are in correspondence with the \((n-1)\)-simplices in the associated flag complex. The \(n\)-clique communities are obtained from the \(Q\)-analytical information of the flag complex. Note that we can put the \(n\)-cliques as vertices of a graph with edges between vertices if the associated cliques share \(n-1\) vertices. This leads us to define our first connectivity graph.
\begin{definition}\label{def:q-graph}
	The \(q\)-\emph{graph} of a simplicial complex \(K\) has as its vertices the simplices in \(K_q\) and edges between pairs of \(q\)-near simplices. 
\end{definition}

Standard \(Q\)-analysis as outlined above fails to take into account directionality in the case of digraphs and directed flag complexes. Consider the cycle and star digraphs in the figure below. As undirected graphs, or 1-dimensional simplicial complexes, they are indistinguishable by \(q\)-connectivity information: both contain one component so they are 0-connected, and the maximal 1-simplices each form their own 1-connected components in both cases.
\begin{center}
	\begin{tikzpicture}
	\tikzstyle{point}=[circle,thick,draw=black,fill=black,inner sep=0pt,minimum width=2pt,minimum height=2pt]
	\tikzstyle{arc}=[shorten >= 8pt,shorten <= 8pt,->, draw=red, thick]
	
	\node[] at (-1.2,-1) {1};
	\node[] at (1.2,-1) {2};
	\node[] at (0,1) {0};
	
	\draw[arc] (0,1) to (-1.2,-1);
	\draw[arc] (-1.2,-1) to (1.2,-1);
	\draw[arc] (1.2,-1) to (0,1);
	
	\node[] at (5,-0.3) {a};
	\node[] at (3.8,-1) {b};
	\node[] at (6.2,-1) {c};
	\node[] at (5,1) {d};		
	
	\draw[arc] (3.8,-1) to (5,-0.3);
	\draw[arc] (6.2,-1) to (5,-0.3);
	\draw[arc] (5,1) to (5,-0.3);
	\end{tikzpicture}
\end{center}
We therefore introduce a refined version of \(Q\)-analysis that is sensitive to the directionality of simplices, or directed cliques in directed graphs. We do this by imposing directions through face maps.
\begin{definition}\label{def:di_hat}
	Let \(\sigma\) be an \(n\)-simplex. We denote by \(\widehat{d_i}\) the face map
	\[\widehat{d_i}(\sigma) = 
	\begin{cases}
	(v_0,\dots,\widehat{v_i},\dots, v_n), \ \text{if } i < n,\\
	(v_0,\dots,v_{n-1}, \widehat{v_n}), \ \text{if } i \geq n.
	\end{cases}\]
\end{definition}
The face map \(\widehat{d_i}\) now makes sense in any dimension since it always removes the vertex at position \(\text{min}\{i,\text{dim}(v)\}\). The reason to modify the standard face map in this fashion is due to the fact \(q\)-connectivity looks at all the simplices of dimension \(q\) and higher.
\begin{definition}\label{def:q_near_directed}
	For an ordered simplicial complex \(K\), let \((\sigma,\tau)\) be an ordered pair of simplices \(\sigma \in K_s\) and \(\tau \in K_t\), where \(s,t \geq q\). Let \((\widehat{d_i},\widehat{d_j})\) be an ordered pair of the \(i\)- and \(j\)-face maps. Then \((\sigma,\tau)\) is \emph{\(q\)-near along \((\widehat{d_i},\widehat{d_j})\)} if either of the following conditions is true:
	\begin{enumerate}
		\item \(\sigma \hookrightarrow \tau,\)
		\item \(\widehat{d_i}(\sigma) \hookleftarrow \alpha \hookrightarrow \widehat{d_j}(\tau),\) for some \(q\)-simplex \(\alpha \in K\).
	\end{enumerate}
	The ordered pair \((\sigma,\tau)\) of simplices of \(K\) is \emph{\(q\)-connected along \((\widehat{d_i},\widehat{d_j})\)} if there is a sequence of simplices in \(K\),
	\[\sigma=\alpha_0,\alpha_1,\alpha_2,\dots,\alpha_n,\alpha_{n+1}=\tau,\]
	such that any two consecutive ones are \(q\)-near along \((\widehat{d_i},\widehat{d_j})\). The sequence of simplices is called a \emph{\(q\)-connection along \((\widehat{d_i},\widehat{d_j})\)} between \(\sigma\) and \(\tau\). We simply write this connection as \((\sigma \alpha_1 \alpha_2 \dots \alpha_n \tau)\).
\end{definition}
We will call the above connection \((q,\widehat{d_i},\widehat{d_j})\)-connection, when the choices of \(q\) and directions \(\widehat{d_i}\) and \(\widehat{d_j}\) are made, and similarly we say \((q,\widehat{d_i},\widehat{d_j})\)-near. 

The directed \((q,\widehat{d_i},\widehat{d_j})\)-connectivity is a preorder on the set of directed cliques \(K_q\). By the classical Alexandroff correspondence, preorders and topological spaces are in bijection \cite{Barmak_book}. The \((q,\widehat{d_i},\widehat{d_j})\)-preorders thus endow a directed graph with a collection of new topological spaces. Up to homotopy it is enough to study partial orders obtained by condensing the preorders (Definition \ref{def:condensation} and the discussion after). The homotopy types of partial orders can then be studied through their order complexes, i.e.\ taking the nerve. 

As our main connectivity digraph stemming from \(q\)-connectivity we take the Hasse diagram form of the \((q,\widehat{d_i},\widehat{d_j})\)-connectivity preorder.
\begin{definition}\
    For an ordered simplicial complex \(K\) the vertices in the \((q,\widehat{d_i},\widehat{d_j})\)-digraph, or simply \(q\)-digraph, are the simplices in \(K_q\), and for two vertices \(\sigma\) and \(\tau\) there is a directed edge \((\sigma,\tau)\) when the pair \((\sigma,\tau)\) is \((q,\widehat{d_i},\widehat{d_j})\)-near.
\end{definition}
As an illustrative example we see that our new connectivity digraph can be used to distinguish prior Example \ref{exspheres}; we refer the reader to \cite{Riihimaki_simplicial_connectivities} for a more detailed investigations of these connectivity digraphs.
\begin{example}\label{ex:q_graphs_spheres}
The full \((1,\widehat{d_1},\widehat{d_2})\)-digraphs of the 2-spheres in Example \ref{exspheres} are shown below.
	\begin{center}
		\begin{tikzpicture}
		\tikzstyle{point}=[circle,thick,draw=black,fill=black,inner sep=0pt,minimum width=2pt,minimum height=2pt]
		\tikzstyle{arc}=[shorten >= 12pt,shorten <= 12pt,->, thick]
		
		\coordinate (01) at (1,1.75);
		\coordinate (02) at (1,1);
		\coordinate (13) at (1,-2.5);
		\coordinate (23) at (1,-3.25);
		\coordinate (12) at (-1.25,-0.75);
		\coordinate (21) at (3.25,-0.75);
		\coordinate (012) at (0,0);
		\coordinate (021) at (2,0);
		\coordinate (123) at (0,-1.5);
		\coordinate (213) at (2,-1.5);		
		
		\node[] at (01) {(01)};
		\node[] at (02) {(02)};
		\node[] at (13) {(13)};
		\node[] at (23) {(23)};
		\node[] at (12) {(12)};
		\node[] at (21) {(21)};
		\node[] at (012) {(012)};
		\node[] at (021) {(021)};
		\node[] at (123) {(123)};
		\node[] at (213) {(213)};
		
		\draw[arc] (012) to (123);
		\draw[arc] (021) to (213);
		\draw[arc] (02) to (012);
		\draw[arc] (02) to (021);
		\draw[arc] (12) to (012);
		\draw[arc] (12) to (123);	
		\draw[arc] (21) to (021);
		\draw[arc] (21) to (213);
		\draw[arc] (13) to (123);
		\draw[arc] (13) to (213);
		\draw[arc,bend right=20] (01) to (012);
		\draw[arc,bend left=20] (01) to (021);
		\draw[arc,bend left=20] (23) to (123);
		\draw[arc,bend right=20] (23) to (213);
		\draw[arc,bend right=20] (021) to (012);
		\draw[arc,bend right=20] (012) to (021);
		
		\coordinate (01) at (7.5,1.75);
		\coordinate (02) at (7.5,1);
		\coordinate (31) at (7.5,-2.5);
		\coordinate (32) at (7.5,-3.25);
		\coordinate (12) at (5.25,-0.75);
		\coordinate (21) at (9.75,-0.75);
		\coordinate (012) at (6.5,0);
		\coordinate (021) at (8.5,0);
		\coordinate (312) at (6.5,-1.5);
		\coordinate (321) at (8.5,-1.5);
		
		\node[] at (01) {(01)};
		\node[] at (02) {(02)};
		\node[] at (31) {(31)};
		\node[] at (32) {(32)};
		\node[] at (12) {(12)};
		\node[] at (21) {(21)};
		\node[] at (012) {(012)};
		\node[] at (021) {(021)};
		\node[] at (312) {(312)};
		\node[] at (321) {(321)};
		
		\draw[arc] (02) to (012);
		\draw[arc] (02) to (021);
		\draw[arc] (12) to (012);
		\draw[arc] (12) to (312);
		\draw[arc] (21) to (021);
		\draw[arc] (21) to (321);	
		\draw[arc] (31) to (312);
		\draw[arc] (31) to (321);
		\draw[arc,bend right=20] (01) to (012);
		\draw[arc,bend left=20] (01) to (021);
		\draw[arc,bend left=20] (32) to (312);
		\draw[arc,bend right=20] (32) to (321);
		
		\draw[shorten >= 18pt,shorten <= 18pt,->, thick,bend right=20] (021) to (012);
		\draw[shorten >= 18pt,shorten <= 18pt,->, thick,bend right=20] (012) to (021);
		\draw[shorten >= 18pt,shorten <= 18pt,->, thick,bend right=20] (321) to (312);
		\draw[shorten >= 18pt,shorten <= 18pt,->, thick,bend right=20] (312) to (321);
		\end{tikzpicture}
	\end{center}
	The connectivity digraphs are different between the spheres. Passing to condensations and order complexes of the associated \((1,\widehat{d_1},\widehat{d_2})\)-preorders, the homotopy type of the left sphere is a wedge of circles \(S^1 \vee S^1\), while that of the right sphere is \(S^1\). The \(q\)-connectivity therefore assigns the underlying digraphs with new homotopy types that distinguish them.
\end{example}

\subsection{The \emph{n}-path digraph}\label{secnpathgr}
Our second example of connectivity digraphs is given by the \emph{$n$-path digraphs} $\{\mathcal{PG}^{(n)}\}_{n\in \N}$. These are defined as the digraphs with  the ordered $(n+1)$-cliques of $\GG$ as vertices, and with directed edges given by their incidence relations (see Definition~\ref{npdigraph}). The construction is a generalization of the line digraph (Definition \ref{deflinegr}); furthermore, we will show that it can be promoted to an endofunctor on the category of digraphs~$\mathbf{Digraph}$ -- \emph{cf.}~Theorem~\ref{functoriality}.

Let $n$ be a positive natural number and let $\GG$ be a directed graph. We start by rephrasing the concept of $q$-graph of Definition~\ref{def:q-graph}, in the specific case in which the simplicial complex~$K$ is  the directed flag complex of a digraph~$\GG$, the simplices have all the same fixed dimension, and the relation is the $1$-codimensional \(q\)-nearness.

\begin{definition}\label{npgraph2}
	Let $\GG$ be a digraph and let $\mathrm{dFl}(\GG)$ be its associated directed flag complex.
	The \emph{$n$-path graph} $\mathcal{G}^{(n)}$ associated to $\GG$ is the graph whose vertices are the $n$-simplices of $\mathrm{dFl}(\GG)$, and such that two vertices $\sigma$ and $\tau$ are connected by an  edge whenever $\sigma$ and $\tau$ share a common $(n-1)$-face.
\end{definition}
\begin{observation}\label{lem:npgr}
Note that Definition \ref{npgraph2} gives the underlying graph for the \(n\)-clique communities (Definition~\ref{def:conncliques}). The name {$n$-path graph} is then justified by the fact that simple paths in $\mathcal{G}^{(n)}$ correspond to ordered $(n+1)$-cliques of $\GG$, consecutively connected by common ordered \(n\)-cliques. If $\GG$ is a digraph, then the $1$-path graph $\mathcal{G}^{(1)}$ associated to $\GG$ is nothing but the line graph $\mathcal{L}(\GG)$ of the underlying undirected graph associated to $\GG$ -- cf.~Definition~\ref{deflinegr}.
\end{observation}

Let $\GG$ be a digraph, $n$ be a natural number  and $\mathrm{dFl}(\GG)$ be the associated directed flag complex. Based on Definition \ref{npgraph2}, we now define the $n$-path digraphs as the connectivity digraphs on the set of ordered $(n+1)$-cliques with their incidence relations. The digraph structure is induced from the total order on $\{0,\dots,n\}$ which induces a total order on the associated face maps.

\begin{definition}\label{npdigraph}
	For $n\geq 1$, the \emph{$n$-path digraph} $\mathcal{PG}^{(n)}$ associated to $\GG$ is the directed graph with the $n$-simplices of~$\mathrm{dFl}(\GG)$ as vertices. For vertices $\sigma$ and $\tau$, there is a directed edge $(\sigma,\tau)$ if, and only if, there is an $(n-1)$-simplex~$\alpha$ of~$\mathrm{dFl}(\GG)$ and some \(i,j \in \{0,\dots,n\}\) such that  
	$$d_i(\sigma) = \alpha = d_j(\tau)  \text{, with } i < j.$$ 
	When $n=0$, we set the $0$-th path digraph $\mathcal{PG}^{(0)}$ to be the digraph $ \GG$.
\end{definition}

We investigate some  properties of the $n$-path digraphs; first, note that these path digraphs generalize the notion of line digraphs:

\begin{proposition}\label{obs1pathareline}
	When $n=1$, the $1$-path digraph $\mathcal{PG}^{(1)}$ is isomorphic to the line digraph $\mathcal{L}(\GG)$ of $\GG$.
\end{proposition}
\begin{proof}
When $n=1$, the vertices of $\mathcal{PG}^{(1)}$ are the edges of $\GG$. The face map $d_1$ applied to an edge $e$ of $\GG$, gives the source of $e$: $d_1(e)=s(e)$. Analogously, we have the relation $d_0(e)=t(e)$. Consequently, two edges $e$ and $f$ of $\GG$ are connected in $\mathcal{PG}^{(1)}$ by a directed edge $(e,f)$ if, and only if, they share a common vertex $d_0(e)=d_1(f)$ in $\GG$. Then the two constructions in Definition~\ref{deflinegr} and in Definition~\ref{npdigraph} are equivalent. 
\end{proof}	

For a digraph $\GG$, denote by $\mathrm{Cone}(\GG)$ the cone digraph obtained from $\GG$ by adding a new vertex~$v_P$ and, for each vertex $v$ in $\GG$, a new directed edge $(v,v_P)$; see Figure \ref{fig:polycone} for an illustration.

	\begin{figure}[h]
		\centering
		\newdimen\R
		\R=2.0cm
		\begin{tikzpicture}
			\draw[xshift=5.0\R, fill] (270:\R) circle(.05)  node[below] {$v_n$};
			\draw[xshift=5.0\R,fill] (225:\R) circle(.05)  node[below left]   {$v_1$};
			\draw[xshift=5.0\R,fill] (180:\R) circle(.05)  node[left] {$v_2$};
			\draw[xshift=5.0\R,fill] (135:\R) circle(.05)  node[above left] {$v_3$};
			\draw[xshift=5.0\R, fill] (90:\R) circle(.05)  node[above] {$v_4$};
			\draw[xshift=5.0\R,fill] (45:\R) circle(.05)  node[above right] {$v_5$};
			\draw[xshift=5.0\R,fill] (0:\R) circle(.05)  node[right] {$v_6$};
			\draw[xshift=5.0\R,fill] (315:\R) circle(.05)  node[below right] {$v_{n-1}$};
			\draw[xshift=5.0\R,fill] (0,0) circle(.03)  node[above right] {};
			
			\node[xshift=5.0\R] (p) at (0,0) { };
			\node[xshift=5.0\R] (v0) at (270:\R) { };
			\node[xshift=5.0\R] (v1) at (225:\R) { };
			\node[xshift=5.0\R] (v2) at (180:\R) { };
			\node[xshift=5.0\R] (v3) at (135:\R) { };
			\node[xshift=5.0\R] (v4) at (90:\R) { };
			\node[xshift=5.0\R] (v5) at (45:\R) { };
			\node[xshift=5.0\R] (v6) at (0:\R) { };
			\node[xshift=5.0\R] (vn) at (315:\R) { };
			
			\draw[thick, red, -latex] (v0)--(v1);
			\draw[thick, red, -latex] (v1)--(v2);
			\draw[thick, red, -latex] (v2)--(v3);
			\draw[thick, red, -latex] (v3)--(v4);
			\draw[thick, red, -latex] (v4)--(v5);
			\draw[thick, red, -latex] (v5)--(v6);
			\draw[thick, red, -latex] (vn)--(v0);

			\draw[thick, blue, -latex] (v0)--(p);			
			\draw[thick, blue, -latex] (vn)--(p);
			\draw[thick, blue, -latex] (v1)--(p);			
			\draw[thick, blue, -latex] (v2)--(p);
			\draw[thick, blue, -latex] (v3)--(p);			
			\draw[thick, blue, -latex] (v4)--(p);
			\draw[thick, blue, -latex] (v5)--(p);			
			\draw[thick, blue, -latex] (v6)--(p);
			
			\draw[xshift=5.0\R, fill] (292.5:\R) node[below right] {$e_{n-1}$};
			\draw[xshift=5.0\R,fill] (247.5:\R) node[below left] {$e_n$};
			\draw[xshift=5.0\R,fill] (202.5:\R)   node[left] {$e_1$};
			\draw[xshift=5.0\R,fill] (157.5:\R)  node[above left] {$e_2$};
			\draw[xshift=5.0\R, fill] (112.5:\R)   node[above] {$e_3$};
			\draw[xshift=5.0\R,fill] (67.5:\R) node[above right] {$e_4$};
			\draw[xshift=5.0\R,fill] (22.5:\R) node[right] {$e_5$};
			\draw[xshift=4.95\R,fill] (337.5:\R)  node {$\cdot$} ;
			\draw[xshift=4.95\R,fill] (333:\R)  node {$\cdot$} ;
			\draw[xshift=4.95\R,fill] (342:\R)  node {$\cdot$} ;
		\end{tikzpicture}
		\caption{The cone $\mathrm{Cone}(C_n)$ of the coherently oriented cyclic digraph $C_n$.} 
		\label{fig:polycone}
	\end{figure}
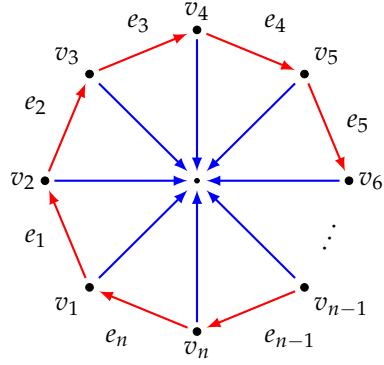

\begin{proposition}\label{prop:cone}
Let $C_n$ be the coherently oriented cyclic digraph on $n$ vertices, with $n\geq 3$. 
Then the the $2$-path digraph of the cone $\mathrm{Cone}(C_n)$  is isomorphic  to $C_n$. \end{proposition}
\begin{proof}
For $n\geq 3$, if $C_n$ is the cyclic digraph on $n$ vertices with all the edges coherently oriented, then the cone~$\mathrm{Cone}(C_n)$ has $2$-simplices based at the edges of $C_n$. 
For each edge $(v_i,v_{i+1})$ in~$C_n$, the edge $(v_i,v_P)$ of  $\mathrm{Cone}(C_n)$ can be written as $d_1[v_i,v_{i+1},v_P]=d_0[v_{i-1},v_{i},v_P]$ (where the indices $i$ are taken modulo $n$). Then, it is easy to check that the associated $2$-path digraph is isomorphic to $C_n$.
\end{proof}

The result generalizes by induction to every $m$: let $\mathrm{Cone}^m(\GG)$ denote the $m$-th iterated cone of $\GG$, i.e. $\mathrm{Cone}^m\coloneqq \mathrm{Cone}\circ \dots \circ \mathrm{Cone}$, $m$ times. Then, we have the following:

\begin{proposition}Let $C_n$ be the coherently oriented cyclic digraph on $n$ vertices, with $n\geq 3$. Then,
	\[
	\mathcal{P}^{(m)}\mathrm{Cone}^m(C_n)\cong C_n,
	\]
	the  $m$-path digraph of the $m$-th cone $\mathrm{Cone}^m(C_n)$ of $C_n$  is isomorphic (as a directed graph) to $C_n$. 
\end{proposition} 

We have seen that, for every $m$, the $m$-path digraph can be an oriented cycle. Cycles of ordered $n$-cliques have a rigid structure, as each subsequent element in the sequence is determined by the preceding one, and by the face maps:

\begin{remark}\label{rem:rigid}
Let $\sigma$ and $\sigma'$ be two $n$-simplices of~$\mathrm{dFl}(\GG)$, and assume $d_i(\sigma)=d_j(\sigma')$
for $i<j$, and~$n>0$. If $\sigma=[v_0,\dots,v_n]$, denote by $\sigma[h]$ the $h$-entry $v_h$ of $\sigma$. Then we have 
\[
\sigma'[h]=
\begin{cases}
\sigma[h] & \text{for } h<i,\ h> j \\
\sigma[h+1] & \text{for } i\leq h < j
\end{cases}
\]
and $\sigma'[j]$ is the  vertex in which $\sigma$ and $\sigma'$ differ.  
For example, assume $\sigma=[0,1,2,3]$ and $\sigma'=[1,4,2,3]$; then $d_0(\sigma)=d_1(\sigma')$ and here the face maps correspond to the indices $i=0$, and $j=1$. When $h=i=0$, we have $\sigma'[0]=\sigma[1]$, and $\sigma'[h]=\sigma[h]$ for $h=2,3$; on the other hand, when  $h=j$ (case $h=j=1$), we have $\sigma'[h]=4$, the  vertex in which $\sigma$ and $\sigma'$ differ.
\end{remark}

Such rigidity implies that taking $n$-path digraphs preserves acyclicity:

\begin{proposition}\label{prop:orcl}
	The $n$-path digraph $\mathcal{PG}^{(n)}$ of a digraph $\GG$ without oriented cycles does not have oriented cycles.
\end{proposition}

\begin{proof}
	We proceed by contradiction.
	Assume  $\mathcal{PG}^{(n)}$ has an oriented cycle~$\gamma$ given by $n$-simplices 
	\begin{equation}\label{eq:cycle}
	  \sigma_0\to\sigma_1\to\dots\to\sigma_{k-1}\to\sigma_{k}
	\end{equation}
	of $\mathrm{dFl}(\GG)$. All the discussion below is given modulo $k$.
	
	The oriented cycle in Equation~\eqref{eq:cycle} corresponds to a closed path of ordered $(n+1)$-cliques of~$\GG$, and for each $\sigma_h\to\sigma_{h+1}$ there are indices $i_h$ and $j_{h}$ such that $d_{i_h}(\sigma_h)=d_{j_h}(\sigma_{h+1})$, with  $i_h<j_{h}$. Without loss of generality assume that $i_h=0$ for some $h$ -- otherwise let $i=\min\{i_h\mid h=0,\dots,k\}$ and replace in the discussion below the index~$0$ with such minimal index~$i$. 
	
	 Starting with the cycle~$\gamma$, it is now possible to construct oriented closed paths in $\GG$ as follows. First, as  $i_h=0$ for some $h$, then we have $\sigma_h[1]=\sigma_{h+1}[0]$ by Remark~\ref{rem:rigid}. Furthermore, as the simplex~$\sigma_h$ represents an ordered clique of $\GG$, this means that there is a directed edge~$e_h$ between the $0$-th and $1$-st entry of $\sigma_h$, i.e. $e_h\coloneqq (\sigma_h[0],\sigma_h[1])$, and we can see $e_h$ as an edge between the vertices $\sigma_h[0]$ and 	$\sigma_{h+1}[0]$. The idea is now to use these edges~$e_h$ to construct a cycle $\gamma_0$ in~$\GG$. To this goal, consider only the indices $h$ in $\{0,\dots,k\}$ for which~$i_h=0$, say~${h_0},\dots,{h_s}$. For all other indices~$r$, we have~$\sigma_r[0]=\sigma_{h_r+1}[0]$, where $h_r\coloneqq \min_{{h_0},\dots,{h_s}}\{h_j<r\}$. Then, starting with $v_0\coloneqq\sigma_0[0]=\sigma_{{h_0}}[0]$, we have the sequence of edges
	$$
	v_0=\sigma_{{h_0}}[0]\xrightarrow{e_{{h_0}}}\sigma_{{h_0}+1}[0]=\sigma_{h_1}[0]\xrightarrow{e_{{h_1}}}\sigma_{{h_1}+1}[0]\xrightarrow{}\dots\xrightarrow{e_{{h_s}}}\sigma_{{h_s}+1}[0]=\sigma_{{h_0}}[0]=v_0
	$$
	terminating again in $v_0$, as $\gamma$ was a cycle of simplices. But, this is not possible because $\GG$ has no oriented cycles, leading to a contradiction.
\end{proof}

In the next example we show that the directed structure inherited by the path digraphs is more informative than the undirected one. We
 apply the constructions of Definition~\ref{npgraph2} and   of Definition~\ref{npdigraph} to the digraphs shown in Example~\ref{expaths}:

\begin{example}\label{exHH2}
	Let $\GG_1$ and $\GG_2$ be the graphs of Example~\ref{expaths}.   
The $2$-path graphs $\mathcal{G}_1^{(2)}$ and $\mathcal{G}_2^{(2)}$ have three vertices, corresponding to the three $2$-simplices, and two edges, corresponding to the two edges in common; the obtained path graphs are isomorphic. 
On the other hand, the associated $2$-path \emph{digraphs} $\mathcal{PG}_1^{(2)}$ and $\mathcal{PG}_1^{(2)}$ are not isomorphic. In fact, we have
	\[			
	\mathcal{PG}_1^{(2)}\cong
	\begin{tikzpicture}[baseline=(current bounding box.center)]
		
		\tikzset{vertex/.style = { auto}}
		\tikzset{edge/.style = {->,> = latex'}}
		\node[vertex] (0) at  (-2,0) {$\bullet$};
		\node[vertex] (1) at  (0,0) {$\bullet$};
		\node[vertex] (2) at  (2,0) {$\bullet$};
		\draw[edge] (1) to (2);
	\end{tikzpicture}
	\]
	\[			
	\mathcal{PG}_2^{(2)}\cong
	\begin{tikzpicture}[baseline=(current bounding box.center)]
		
		\tikzset{vertex/.style = { auto}}
		\tikzset{edge/.style = {->,> = latex'}}
		\node[vertex] (0) at  (-2,0) {$\bullet$};
		\node[vertex] (1) at  (0,0) {$\bullet$};
		\node[vertex] (2) at  (2,0) {$\bullet$};
		\draw[edge] (1) to (2);
		\draw[edge] (0) to (1);
	\end{tikzpicture}
	\]
	 showing that the extension from 2-path graphs to 2-path digraphs provides additional non-trivial information.
\end{example}

We finish by establishing our main result concerning \(n\)-path digraphs, which is their functoriality.

\begin{theorem}\label{functoriality}
	For each $n$ in $\N$,  $$\mathcal{P}^{(n)}\colon \mathbf{DiGraph}\to\mathbf{DiGraph}, \quad \GG\mapsto \mathcal{PG}^{(n)}$$   is an endofunctor of the category  of directed graphs.
\end{theorem}
	
	\begin{proof}
	First, note that if $n=0$, then $\mathcal{P}^{(n)}$ is the identity functor by definition; if $n=1$, then, by Proposition~\ref{obs1pathareline}, $\mathcal{P}^{(n)}$ coincides with the line digraph functor (see Remark~\ref{linefunctor}).  
	
	Let now $n \geq 2$. We first observe that, if a digraph $\GG$ has no ordered $(n+1)$-cliques, then the associated $n$-path digraph is the empty digraph. 
	By Remark~\ref{obsmorgr}, a morphism of digraphs $\phi\colon \GG_1\to\GG_2$
	induces a function between the sets of $(n+1)$-cliques. As a consequence, if $\GG_1$ and $\GG_2$ are two digraphs and  the set of $(n+1)$-cliques of $\GG_2$ is empty, then the set $\mathrm{Hom}_{\mathbf{Digraph}}(\GG_1,\GG_2)$ of morphisms of digraphs between $\GG_1$ and $\GG_2$ is also empty. 
	By Remark~\ref{rem:digraph_morphism_on_cliques} we also have an induced function between the sets of ordered $(n+1)$-cliques of $\GG_1$ and $\GG_2$, that preserves the relative order of the faces. This function may not be surjective (there might be cliques that are not images of any clique in \(\GG_1\)) nor injective (as it may send different cliques to the same one).
	
	For a morphism of digraphs $\phi\colon \GG_1\to\GG_2$, we have gotten a function $\mathcal{P}^{(n)}(\phi)$ between the sets of vertices of the associated $n$-path digraphs. We now want to promote it to a morphism of $n$-path digraphs.
	To this end, let $c,c'$ be two ordered $(n+1)$-cliques of $\GG_1$, and let  $\sigma_c$ and $\sigma_{c'}$ be the associated $n$-simplices in  $\mathrm{dFl}(\GG_1)$. It may happen that  $\mathcal{P}^{(n)}(\phi)$ sends both $\sigma_c$ and $\sigma_{c'}$ to the same simplex of $\mathrm{dFl}(\GG_2)$. However, observe that if $\sigma_c$ and $\sigma_{c'}$ share an $(n-1)$-face $\tau$, and are sent to the same $n$-simplex of $\mathrm{dFl}(\GG_2)$, then the ordered $(n+1)$-cliques $c,c'$ share the face~$\tau$ such that there exists an index~$i$ with $d_i(\sigma_c) = d_i(\sigma_{c'})=\tau$ (as the linear ordering should be preserved).
	Therefore, if $\sigma_c$ and $\sigma_{c'}$
	are collapsed to the same vertex of  $\mathcal{PG}_2^{(n)}$, then $\sigma_c$ and $\sigma_{c'}$ represent two non-adjacent vertices of $\mathcal{PG}_1^{(n)}$. Furthermore, the relative incidence relations are preserved, and the morphism of digraphs $\phi\colon \GG_1\to\GG_2$ induces a function $\mathcal{P}^{(n)}(\phi)$  such that $(\mathcal{P}^{(n)}(\phi)(\sigma_c),\mathcal{P}^{(n)}(\phi)(\sigma_{c'}))\in \mathcal{PG}_2^{(n)}$ for every $(\sigma_c,\sigma_{c'})$ belonging to $\mathcal{PG}_1^{(n)}$, i.e. $\mathcal{P}^{(n)}(\phi)$ is a morphism of digraphs.
	
To conclude, it is now easy to see that $\mathcal{P}^{(n)}$ of the identity is the identity and that, if $\phi_1$ and $\phi_2$ are morphisms of digraphs, then $\mathcal{P}^{(n)}(\phi_1\circ\phi_2)=\mathcal{P}^{(n)}(\phi_1)\circ\mathcal{P}^{(n)}(\phi_2)$.   This shows that $\mathcal{P}^{(n)}$ is an endofunctor of the category $\mathbf{Digraph}$.
	\end{proof}

By Proposition~\ref{prop:orcl}, we get also the functoriality when restricting to acyclic digraphs:

\begin{proposition}
	 The funtor $\mathcal{P}^{(n)}$ restricts to a functor $\mathbf{DiGraph}_0\to\mathbf{DiGraph}_0$ on the category of finite digraphs without oriented cycles. 
\end{proposition}

\begin{remark}\label{rem:relax}
The idea behind Definition~\ref{npdigraph} corresponds to the intuition that flows in a directed graph follow the direction of the edges, from the source to the target. As we have remarked in Lemma~\ref{obs1pathareline}, the source and target of a directed edge are given by the face maps $d_1$ and $d_0$, respectively. The condition  $i<j$ in the construction of the path digraph follows and generalizes this principle to the higher simplices as well. This condition can be relaxed to $i\leq j$, which has the effect that the path digraphs might have reciprocal edges, and therefore oriented cycles.
\end{remark}

As described in Section~\ref{secflcompl}, one of the possible approaches to an homology theory of digraphs is given by the ordinary homology of the associated directed flag complexes; recall that this is constructed by using the set of ordered cliques in a digraph. When applied to $1$-path digraphs, we have the following consequence:

\begin{remark}
   Let $\GG$ be a digraph without oriented cycles. Then,  the directed flag complex $\mathrm{dFl}(\mathcal{PG}^{(1)}(\GG))=\mathrm{dFl}(\mathcal{L}(\GG))$ of the $1$-path digraph has simplices of dimension at most $1$. 
\end{remark}

The above remark is not true for $n$-path digraphs. For example,  it is easy to see that the $2$-path digraphs may contain $3$-cliques, and as a consequence the associated directed flag complexes can possibly be of dimension at least $2$:

\begin{example}\label{ex:2clfl}
Consider the digraph $\GG$ on five vertices with directed edges as follows:
	\begin{center}
		\begin{tikzpicture}
			\tikzstyle{arc}=[shorten >= 3pt, red, shorten <= 3pt,->, thick]
			
			\node[] (0) at  (0,0) {0};
			\node[] (1) at  (0,1.5) {1};
			\node[] (2) at  (1.5,0) {2};
			\node[] (3) at  (1.5,1.5) {$3$};
			\node[] (4) at  (2.5,2) {4};
			
			\draw[arc] (0) to (1);
			\draw[arc] (0) to (2);
			\draw[arc] (1) to (3);
			\draw[arc] (1) to (2);
			\draw[arc] (2) to (3);
			\draw[arc] (1) to[bend left] (4);
			\draw[arc] (4) to[bend left] (2);
	\end{tikzpicture}
	\end{center}
	Then, $\GG$ contains three ordered cliques (corresponding to the simplices $[0,1,2]$, $[1,2,3]$ and $[1,4,2]$. The boundary relations show that the associated $2$-path digraph is the ordered clique on three vertices, and the associated directed flag complex is a $2$-simplex.
	\end{example}
	
{In the case of \(q\)-connectivity the homotopy types of the connectivity digraphs can be studied through the order complex construction, recall Example \ref{ex:q_graphs_spheres}.} Analogously, it is then natural to ask what is the dimension of the directed flag complex of an $n$-path digraph, and what is its homotopy type: 

\begin{question}\label{q:htypef}
For a given digraph $\GG$, what is the homotopy type of the directed flag complex associated to $\mathcal{PG}^{(n)}$, or to the relaxed $n$-path digraph of Remark~\ref{rem:relax}? What are the distributions of the associated Betti numbers like?
\end{question}
A partial answer to this is given in \cite{Riihimaki_simplicial_connectivities} when \(\GG\) is the 1-skeleton of so called pseudomanifold; for example, the Cone(\(C_n\)) in Figure \ref{fig:polycone} is an example of a directed 2-pseudomanifold. In this case the directed flag complexes of the connectivity digraphs we have considered are 1-dimensional. The digraph in Example \ref{ex:2clfl} does not have the structure of a 2-pseudomanifold: there is a "singular" edge (1,2) to which three different 2-simplices are attached.
We conclude the section with examples showing that the digraphs of Example~\ref{expaths} and Example~\ref{exspheres} can now be distinguished by using the homology groups associated to the $2$-path digraphs:

\begin{example}\label{exHHdigraps}
Consider the digraphs $\GG_1$ and $\GG_2$ illustrated in Example~\ref{expaths}.
	The associated $2$-path digraph $\mathcal{PG}_1^{(2)}$ -- cf.~Example~\ref{exHH2} -- has two connected components, whereas the $2$-path digraph~$\mathcal{PG}_2^{(2)}$  has one connected component. 
	Then all the homology theories described in Section~\ref{homologies} can now distinguish the two digraphs.
\end{example}

\begin{example}\label{exspheresHH}
	Consider the  digraphs $\mathcal{S}_1$ and $\mathcal{S}_2$ of Example~\ref{exspheres}.  
	The associated $2$-path digraph~$\mathcal{PS}_1^{(2)}$ is  the following:
	\[
		\begin{tikzpicture}[baseline=(current bounding box.center)]
			\tikzset{vertex/.style = {minimum size=1.5em}}
			\tikzset{edge/.style = {->,> = latex'}}
			
			\node[vertex] (0) at  (0,1.5) {$\bullet$}; 
			\node[vertex] (1) at  (2,1.5) {$\bullet$}; 
			\node[vertex] (2) at  (0,0) {$\bullet$}; 
			\node[vertex] (3) at  (2,0) {$\bullet$}; 
			\draw[edge] (0) to (2) ;
			\draw[edge] (1) to (3);
			\draw[edge] (0) to[bend left] (1);
			\draw[edge] (1) to[bend left] (0);
			\draw[edge] (2) to[bend left] (3);
			\draw[edge] (3) to[bend left] (2);
		\end{tikzpicture}
	\]
	The digraph $\mathcal{PS}_2^{(2)}$, instead, is a disconnected digraph:
	\begin{center}
		\begin{tikzpicture}[baseline=(current bounding box.center)]
			\tikzset{vertex/.style = {minimum size=1.5em}}
			\tikzset{edge/.style = {->,> = latex'}}
			
			\node[vertex] (0) at  (0,1.7) {$\bullet$}; 
			\node[vertex] (1) at  (2,1.7) {$\bullet$}; 
			\node[vertex] (2) at  (0,0) {$\bullet$}; 
			\node[vertex] (3) at  (2,0) {$\bullet$}; 
			\draw[edge] (0) to[bend left] (1);
			\draw[edge] (1) to[bend left] (0);
			\draw[edge] (2) to[bend left] (3);
			\draw[edge] (3) to[bend left] (2);
		\end{tikzpicture}
	\end{center}
	with two connected components. Again, all the homology theories described in Section~\ref{homologies} can now distinguish the two digraphs representing 2-spheres.
\end{example}

\section{Persistent Hochschild homology of digraphs}\label{sec:extendedhomoloies}

The goal of this section is to introduce a persistent homology framework for Hochschild homology of directed graphs, using connectivity digraphs as an intermediate step.
The reason why we restrict our analysis to Hochshild homology is simple. As observed in Remark~\ref{obs:flagsforgetdir}, the ordinary homology of (the directed flag complex associated to) a directed graph loses some information given by the directionality. On the other hand, path homology groups in degree $0$ and $1$ are strictly related to the homology of the directed flag complex, and in degrees $i\geq 2$ are of difficult computation. 

One of the disadvantages of Hochschild homology for digraphs is that it is trivial in degrees $i\geq 2$. The use of connectivity digraphs is meant to solve this issue. We first show that the $n$-path digraphs introduced in Section~\ref{secnpathgr} allow us to construct a persistent homology functorially in the case of acyclic digraphs. We then extend the persistence pipeline to general digraphs; we lose functoriality but we obtain a new persistence descriptor for digraphs.

\subsection{Persistent Hochschild homology of acyclic digraphs}\label{sec:PHH}                                 
In this subsection we mainly follow~\cite{vertechi}, where an abstract categorical framework in which to develop persistent homology theories has been introduced. In this framework, one replaces filtrations of topological spaces with filtrations in an arbitrary category (for us, the category of directed graphs) and the homology functors with any functor with values in a regular ranked category (for us, Hochschild homology over a field~$\K$), compare with \cite[Table~1]{vertechi}.

We start by considering  the poset $(\bR, \leq)$ of real numbers with the  induced natural partial order $\leq$. The poset $(\bR, \leq)$ can  be seen as a category in a standard way, as	every poset can be  seen as a category: the category $\mathbf{P}$ associated to a poset $P=(S,\leq)$ has the set $S$ as a collection of objects and a (unique) morphism  $x\to y$ for any $x\leq y$.  A morphism of posets is then a functor between them, equivalently an order-preserving map of posets.

In persistent homology applications, one usually considers diagrams of spaces indexed by the natural numbers, or more generally by the real numbers. These diagrams are referred to as filtrations:

\begin{definition}
	A (real-indexed) \emph{filtration} in a category $\mathbf{C}$ is a functor $\mathcal{F}\colon(\bR, \leq)\to \mathbf{C}$.
\end{definition}

\begin{remark}\label{rem:tameness}
    Following \cite{vertechi} we will always consider \emph{tame} filtrations. Essentially, a filtration \(\mathcal{F}\) is tame if there is a finite sequence \(\{t_i\}_{i \in \N}\) such that \(\mathcal{F}(a) \ra \mathcal{F}(b)\) may fail to be an isomorphism only if \(a < t_i \leq b\) for some \(i\). The concept of tameness extends to subposets of \((\bR, \leq)\).
\end{remark}

\begin{example}
Let $\{\GG_n\}_{n\in \N}$ be a family of digraphs, with $\GG_n\to\GG_{n+1}$ a morphism of digraphs for each $n\in \N$. Then, the family  $\{\GG_n\}_{n\in \N}$ yields a filtration $(\N, \leq) \ra \mathbf{Digraph}$; if we assume $\GG_n$ to be without oriented cycles, the filtration takes values in $\mathbf{Digraph}_0$.  Observe that if $\{\GG_n\}_{n\in \N}$ is a family of subgraphs of a given directed graph $\GG$, then by Remark~\ref{rem:tameness} the resulting filtration is tame.
\end{example}

Let $\mathcal{F}$ be a filtration in $\mathbf{Digraph}_0$. In the following discussion, we restrict to the $n$-path digraph functor~$\mathcal{PG}^{(n)}$, but everything is the same by replacing $\mathcal{PG}^{(n)}$ with any functorial construction of connectivity digraphs. By Theorem~\ref{functoriality}, composition with the $n$-path digraph functor~$\mathcal{PG}^{(n)}$ induces, for each $n$ in $\N$, a filtration in  $\mathbf{Digraph}_0$; by Proposition~\ref{prop:orcl}, the $n$-path digraph of a digraph without oriented cycles does not have oriented cycles. We then get the following composition of functors:
\[
(\bR, \leq)\xrightarrow{\mathcal{F}}\mathbf{Digraph}_0\xrightarrow{\mathcal{P}^{(n)}}\mathbf{Digraph}_0 \ .
\]
Let $\mathbf{FinVect}$ be the category of finite dimensional vector spaces over an algebraically closed field~$\K$.  By Remark~\ref{obs:HHfunctorfin}, the Hochschild homology groups yield  functors with values in $\mathbf{FinVect}$. 

\begin{remark}
The category $\mathbf{FinVect}$, equipped with the dimension function assigning to a vector space its dimension, is a ranked category -- cf.~\cite[Definition~2.1]{vertechi}.
\end{remark}

Before putting all together, we need the definition of a persistence function, generalizing persistent Betti numbers from classical persistent homology:
 
 \begin{definition}\cite[Definition~3.2]{vertechi}\label{def:cat_pers_func}
 	Let $\mathbf{C}$ be a category. An integer-valued lower-bounded function $p$ on the morphisms of $\mathbf{C}$ is a \emph{categorical persistence function} if, for all $u_1\to u_2\to v_1 \to v_2$ the following hold:
 	\begin{enumerate}
 		\item $p(u_1\to v_1)\leq p(u_2\to v_1)$ and $p(u_2\to v_2)\leq p(u_2\to v_1)$;
 		\item $p(u_2\to v_1)- p(u_1\to v_1)\geq p(u_2\to v_2)- p(u_1\to v_2)$.
 	\end{enumerate}
 \end{definition}
 
 Let $\mathcal{F}\colon (\bR , \leq) \to\mathbf{Digraph}_0$ be a filtration, and consider the following composition of functors:
\begin{equation}\label{eq:digraph_filtration}
(\bR, \leq) \xrightarrow{\mathcal{F}}\mathbf{Digraph}_0\xrightarrow{\mathcal{P}^{(n)}}\mathbf{Digraph}_0\xrightarrow{\K-} \K\text{-}\mathbf{Alg}\xrightarrow{\mathrm{HH}_1}\mathbf{FinVect}
\end{equation}
By \cite[Proposition~3.6]{vertechi}, any functor $\mathbf{C}\to \mathbf{FinVect}$ yields a categorical persistence function. In order to get an analogue of persistent Betti numbers usually associated to pairs of real numbers, it is then sufficient to have a filtration and a categorical persistence function \cite[Remark~3.8]{vertechi}. In our context,  $\mathcal{F}\colon (\bR , \leq) \to\mathbf{Digraph}_0$ is a filtration, and the composition of functors in~\eqref{eq:digraph_filtration}
yields the categorical persistence function. Furthermore, for each pair of real numbers $u\leq v$, Hochschild homology gives the linear maps
\[\mathrm{HH}_0(\K\mathcal{PG}_u^{(n)})\to \mathrm{HH}_0(\K\mathcal{PG}_v^{(n)}) \ \text{ and  } \ \mathrm{HH}_1(\K\mathcal{PG}_u^{(n)})\to \mathrm{HH}_1(\K\mathcal{PG}_v^{(n)})\]
of finite dimensional vector spaces. By taking the images of these maps, we get the desired persistent Betti numbers:

\begin{definition}
	The $(n,1)$-persistent Betti number of a filtration in $\mathbf{Digraph}_0$ is the persistent Betti number induced by $\mathrm{HH}_1(\K\mathcal{PG}_u^{(n)})\to \mathrm{HH}_1(\K\mathcal{PG}_v^{(n)})$, in dimension $1$. The $(n,0)$-persistent Betti number for dimension $0$ is defined analogously.
\end{definition}

\begin{remark}
Note that degrees of Hochschild homology are only 0 and 1. The higher "homological" degrees \(n\), and the \(n\)-th Betti numbers, are defined by the connectivity digraphs of \(n\)-simplices. This is in our view the lifting of Hochschild homology beyond degree 1.
\end{remark}

In this functorial framework, given a categorical persistence function $p$ and a (tame) filtration $\mathcal{F}$, it is possible to define a persistence diagram~$D\mathcal{F}$ as well ~\cite{vertechi}. In the case of real-indexed filtrations, the morphisms \(u \leq v \in \bR\) are in bijection with the positive half-plane \(\Delta^+ = \{(u,v) \in \bR^2 \ | \ u \leq v\}\). We therefore get an induced persistence function \(p_\mathcal{F} \colon \Delta^+ \ra \Z\) given by \((u,v) \mapsto p(\mathcal{F}(u \leq v))\). For \(u < v\) we define the multiplicity \(\mu(u,v)\) as
\[\min_{I_u,I_v}\{p_\mathcal{F}(\text{sup}(I_u), \text{inf}(I_v)) - p_\mathcal{F}(\text{inf}(I_u), \text{inf}(I_v)) - p_\mathcal{F}(\text{sup}(I_u), \text{sup}(I_v)) + p_\mathcal{F}(\text{inf}(I_u), \text{sup}(I_v))\},\]
where \(I_u\) and \(I_v\) range over disjoint connected neighborhoods of \(u\) and \(v\). The persistence diagram \(D\mathcal{F}\) associated to \(p_\mathcal{F}\) is then defined by those points \((u,v)\) such that \(\mu(u,v) > 0\), together with the diagonal $\{(u,u)\mid u\in [0,\infty)\}$ \cite[Definition~6]{vertechi}. Note that for small enough neighborhoods \(I_u\) and \(I_v\) we have \(\text{inf}(I_u) \ra \text{sup}(I_u) \ra \text{inf}(I_v) \ra \text{sup}(I_v)\), and the above minimized expression is exactly that of condition 2. in Definition \ref{def:cat_pers_func} with strict inequality.

We get an immediate stability theorem, in fact an isometry theorem, between our Hochschild homology valued filtrations and their persistence diagrams. Let $d_{\mathbf{FinVect}}(\mathcal{F},\mathcal{G})$ be the interleaving distance between two filtrations $\mathcal{F},\mathcal{G} \colon (\bR, \leq) \ra \mathbf{FinVect}$ and let $d(D\mathcal{F},D\mathcal{G})$ be the bottleneck distance between the associated persistence diagrams.

\begin{theorem}\label{thm:stability}
	Let $\mathcal{F,G}\colon (\bR, \leq)\to\mathbf{Digraph}_0$ be two filtrations of digraphs. Then,
	\[
	d_{\mathbf{FinVect}}(\mathrm{HH}_1\circ \K\mathcal{P}^{(n)}\circ\mathcal{F},\mathrm{HH}_1\circ \K\mathcal{P}^{(n)}\circ\mathcal{G})= d(D\mathcal{F},D\mathcal{G})
	\]
	where $D\mathcal{F}$ and $D\mathcal{G}$ represent the persistence diagrams associated to $\mathcal{F}$ and $\mathcal{G}$.
\end{theorem}

\begin{proof}
The proof follows directly from \cite[Theorem~3.27 \& 3.29]{vertechi}.
\end{proof}

The theorem says that the persistent Hochschild homology groups associated to a filtration of digraphs are stable, one of the required properties in persistent homology applications.  
We refer to this as \emph{persistent Hochschild homology}. Note that in the terminology of \cite{Bubenik_metrics} we have introduced a generalised persistence along with a \emph{hard stability theorem}: the map from persistence modules to discrete invariants is 1-Lipschitz. This still leaves open the (very hard) problem of proving a stable persistent Hochschild homology pipeline starting from the space of (acyclic) digraphs. The main obstacle here lies in the difficulty of having an appropriate metric for digraphs that behaves well with filtrations. 

Note that persistent Hochschild homology does not behave as the usual persistent homology; in fact, we have the following:
\begin{remark}\label{rem:Hochclasses}
Consider an edge weighted directed graph $\GG$, and consider the filtration induced by sorting the weights in an increasing order -- the first digraph in the sequence being the spanning subgraph on the vertices of $\GG$. Then, generators in persistent Hochschild homology applied to such filtration always persist until $\infty$. Therefore we cannot talk about births and deaths as is usually done in the context of barcodes. 
\end{remark}

Similarly, as for Hochschild homology,  we can consider compositions 
\[
(\bR, \leq) \xrightarrow{\mathcal{F}}\mathbf{Digraph}\xrightarrow{\mathcal{P}^{(n)}}\mathbf{Digraph}\xrightarrow{\mathrm{Ch}\circ\mathrm{dFl}} 
\mathbf{Ch}\xrightarrow{\mathrm{H}_n} \mathbf{Ab}
 \]
involving the homology of directed flag complexes, or the path homology functors. For $n=0$, these compositions  yield the usual persistence homology in the first case, and the persistent path homology introduced in~\cite{inbook_Chowdhury}, in the second. Henceforth, composition with the higher connectivity digraphs allows us to extend the usual classical pipelines. 

\subsection{A persistent Hochschild homology pipeline for directed graphs}\label{secHHpipe}

As seen in the previous section, Hochschild homology gives rise to a persistence function when considering directed graphs without oriented cycles; however, this approach fails with graphs having oriented cycles, due to Proposition~\ref{prop:HH finite}. We now introduce a persistence-like framework for Hochschild homology that extends our set-up to the whole category $\mathbf{Digraph}$. 

Our pipeline proceeds as follows:
\begin{enumerate}
    \item We start with a filtration~$\mathcal{F}\colon (\bR , \leq) \to\mathbf{Digraph}$.
    \item At each filtration step \(t\) we obtain a connectivity digraph \(E\mathcal{F}_t\) by applying the $n$-path digraph $\mathcal{P}^{(n)}$ for some \(n\), or the \(q\)-graph for some \((q,\widehat{d_i},\widehat{d_j})\).
    \item The resulting digraphs \(E\mathcal{F}_t\) can in general have oriented cycles. We therefore consider the condensation $c(E\mathcal{F}_t)$ (Definition~\ref{def:condensation}).
    \item    By Remark~\ref{rem:condensation wo cycles}, $c(E\mathcal{F}_t)$ does not have oriented cycles. Hence, we compute the Hochschild characteristic $\mathcal{X}_{\mathrm{HH}}(c(E\mathcal{F}_t)) = \dim_k \mathrm{HH}_0(A)-\dim_k \mathrm{HH}_1(A)$, where $A$ is the associated path algebra. Note that $\dim_k \mathrm{HH}_0(A)$ agrees with the number of connected components and $\dim_k \mathrm{HH}_1(A)$ with the formula 	$1-n +\sum_{e\in E(\GG)}\dim_k e_{t(e)}A e_{s(e)}$ of Theorem~\ref{thmhapp}. As the digraph $c(E\mathcal{F}_t)$ does not have oriented cycles, this is exactly the characteristic introduced in Definition~\ref{HHX}.
\end{enumerate}
Diagrammatically, we have:
\begin{equation}\label{eq:compvar}
    (\bR, \leq) \xrightarrow{\mathcal{F}}\mathbf{Digraph}\xrightarrow{E}\mathbf{Digraph}\xrightarrow{c}\mathbf{Digraph}_0\xrightarrow{k-} k\text{-}\mathbf{Alg}\xrightarrow{\mathcal{X}_{\mathrm{HH}}}\mathbf{FinVect} \ .
\end{equation}

 Note that the process lands in the category of finite vector spaces, but the composition is not functorial anymore; we refer to the discussion after Remark \ref{rem:condensation wo cycles}.
\begin{remark}
Observe that taking the condensation of a digraph and the connectivity digraphs do not commute. In particular, killing the cycles in the condensation process may kill also ordered cliques; this is the reason why we first apply $E$ and then the condensation $c$.
\end{remark}

We demonstrate our persistent Hochschild homology pipeline, by applying it to the filtrations of the following digraphs:
\begin{itemize}
    \item Random Erd\"{o}s–Rényi (ER) digraph with probability 0.5 for directed edges between any pair of vertices. We make it randomly edge weighted by replacing each non-zero entry of the adjacency matrix with a value sampled uniformly from \([0,1)\).
    \item Random necklace digraph, as represented in Figure~\ref{fig:necklace}. Similarly, to make it random we first construct the associated adjacency matrix, which has ones on the first upper and lower diagonals, and we then replace these with a value sampled uniformly from \([0,1)\).
    	\begin{figure}[h]
		\centering
		\begin{tikzpicture}[baseline=(current bounding box.center)]
			\tikzset{vertex/.style = {minimum size=1.5em}}
			\tikzset{edge/.style = {->,> = latex',red,thick}}
			
			\node[vertex] (0) at  (0,0) {$\bullet$};
			\node[vertex] (1) at  (1.5,0) {$\bullet$};
			\node[vertex] (2) at  (3,0) {$\bullet$};
			\node[vertex] (4) at  (4.5,0) {...};
			\node[vertex] (5) at  (6,0) {$\bullet$};
			\node[vertex] (6) at  (7.5,0) {$\bullet$};
			\node[vertex] (7) at  (9,0) {$\bullet$};
			\draw[edge] (0) to[bend left] (1);
			\draw[edge] (1) to[bend left] (0);
			\draw[edge] (2) to[bend left] (4);
			\draw[edge] (4) to[bend left] (2);
			\draw[edge] (1) to[bend left] (2);
			\draw[edge] (2) to[bend left] (1);
			\draw[edge] (4) to[bend left] (5);
			\draw[edge] (5) to[bend left] (4);
			\draw[edge] (5) to[bend left] (6);
			\draw[edge] (6) to[bend left] (5);
			\draw[edge] (6) to[bend left] (7);
			\draw[edge] (7) to[bend left] (6);
		\end{tikzpicture}
		\caption{The necklace digraph.}
		\label{fig:necklace}
	\end{figure}
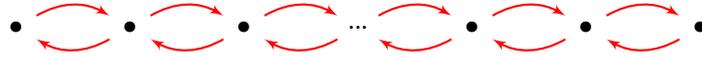
\end{itemize}
In both examples, we take the digraph filtration induced by the random entries of the associated adjacency matrices: at a filtration value \(t\) we take the digraph induced by keeping only edges whose weight is \(\leq t\).  The digraphs we used had 20 vertices, so were represented by $20 \times 20$ adjacency matrices.

Figure \ref{fig:kg} shows the persistent Hochschild characteristics for the random Erd\"{o}s–Rényi and  necklace digraphs, and for their associated 1-path digraphs. 
Figure \ref{fig:q_graph_PHH} shows the results for \(q\)-graphs with \((q,\widehat{d_i},\widehat{d_j})\) equal to \((q,\widehat{d_0},\widehat{d}_{q+1})\) for \(q \in \{1,2,3,4\}\) for the random Erd\"{o}s–Rényi, and \((0,\widehat{d_0},\widehat{d_1})\) for the necklace digraph. 

As observed in Remark~\ref{rem:Hochclasses}, generators in persistent Hochschild homology always persist until $\infty$, and therefore the associated Betti curves yield monotone functions. Experiments involving the persistent Hochschild characteristics of random ER digraphs, and more structured necklace digraphs -- \emph{cf.}~Figure~\ref{fig:kg} -- show instead that the associated curves are not monotone. This is caused by the condensation step in the pipeline introduced in this section. 
Condensation also has the effect of reducing the number of edges and paths in the graph, a number that is correlated to the first Hochschild homology group -- \emph{cf.}~Theorem~\ref{thmhapp}. 
The effect of condensation is seen in the plots as zig-zagging; particularly large positive jumps correspond to large cyclic components being killed.

Note that a common feature in all plots is that the early parts of the filtrations are dominated by the connected components. This occurs until a certain saturation point in which more structured graphs appear and the number of edges, paths, and of oriented cycles, is more prominent. It is interesting to note that this saturation point is reached very soon when dealing with random digraphs and much later for the necklace digraphs. Essentially this is observed in the plots when the value of $\mathcal{X}_{\mathrm{HH}}$ drops to negative. The plots for ER \(q\)-digraphs also begin slightly positive before more paths begin to dominate dropping the values very negatively. Exception is the ER (4,0,5)-digraph: due to the required high-dimensional connecting faces the digraphs are predominantly rather empty of edges and dominated by connected components.

When oriented cycles are more likely to be created with long paths, the variations in $\mathcal{X}_{\mathrm{HH}}$ are stronger. Compare this effect on the persistent Hochschild characteristics of random ER digraphs and of necklace digraphs in Figure \ref{fig:kg}. Necklace digraphs, perturbed with addition of white noise, present small cycles, so that the persistent Hochschild characteristics is changing almost linearly. Changing the associated connectivity digraph may change completely the behaviour of $\mathcal{X}_{\mathrm{HH}}$. Note the change in the persistent Hochschild characteristics associated to the line digraphs of the same random and necklace digraphs. 

Finally, we remark that the $\mathcal{X}_{\mathrm{HH}}$ of the \(q\)-digraphs in Figure \ref{fig:q_graph_PHH} show drastically larger values compared to Figure \ref{fig:kg}. Recall that these digraphs have as vertices all the simplices of dimension \(\geq q\). The simplicial face inclusions are also always near resulting in edges (Definition~\ref{def:q_near_directed}). Therefore the \(q\)-digraphs are larger and more dominated by paths along the filtration. This is particularly visible in the plot for ER (2,0,3)-digraph.

\begin{figure}[h!]
	\begin{center}
	     \begin{subfigure}[b]{0.49\textwidth}
         \centering
		\includegraphics[width=1.0\columnwidth]{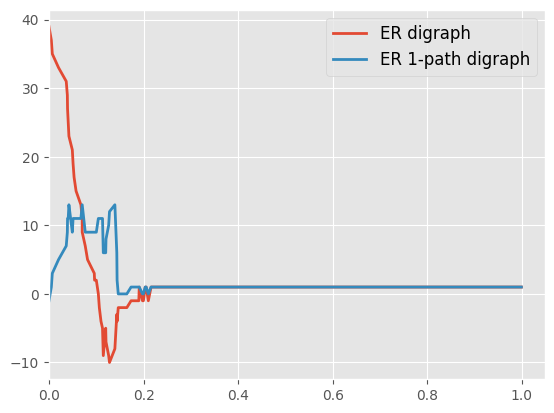}
		     \end{subfigure}
     \hfill
     \begin{subfigure}[b]{0.49\textwidth}
         \centering
		\includegraphics[width=1.0\columnwidth]{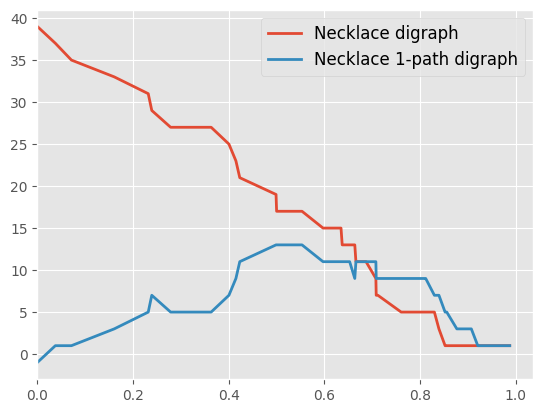}
		\end{subfigure}
	\end{center}
	\caption{Persistent Hochschild characteristics. Left: condensed digraph and condensed $1$-path digraph of a random Erd\"{o}s–Rényi. Right: condensed digraph and condensed $1$-path digraph of a random necklace digraph. Both digraphs had 20 vertices.}
	\label{fig:kg}
\end{figure}

\begin{figure}[h!]
	\begin{center}
	     \begin{subfigure}[b]{0.49\textwidth}
         \centering
		\includegraphics[width=1.0\columnwidth]{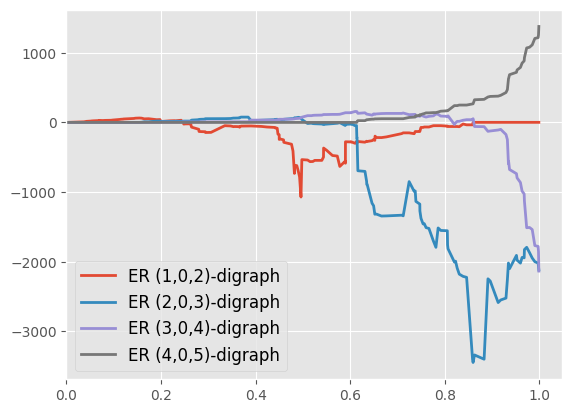}
		     \end{subfigure}
     \hfill
     \begin{subfigure}[b]{0.49\textwidth}
         \centering
		\includegraphics[width=1.0\columnwidth]{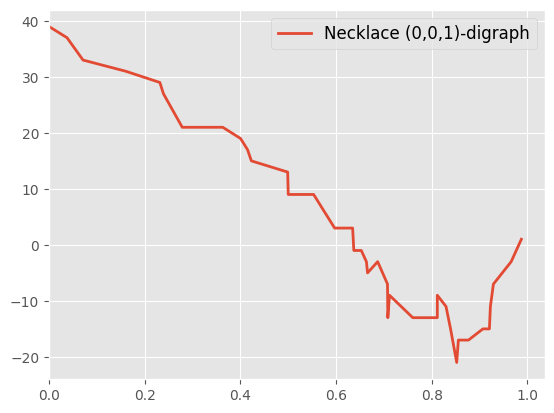}
		\end{subfigure}
	\end{center}
	\caption{Persistent Hochschild characteristics. Left: condensed \(q\)-digraphs of a random Erd\"{o}s–Rényi. Right: condensed \(q\)-digraph of a random necklace digraph. The choices for \((q,\widehat{d_i},\widehat{d_j})\) are shown in the figures. Both digraphs had 20 vertices. }
	\label{fig:q_graph_PHH}
\end{figure}

As already mentioned, the pipeline we have introduced uses the condensation of a digraph to kill the oriented cycles. By producing acyclic digraphs, this also has an effect in the computational efficiency since many graph algorithms, for example finding paths, have lower complexity; this was taken advantage of in \cite{Riihimaki_simplicial_connectivities} in computing \(q\)-connected pathways of simplices. Other approaches are possible, for example in \cite[Algorithm 1]{kaul}  a Berger and Shor algorithm~\cite{10.5555/320176.320203} has been used for the same task. Using condensation leads to our pipeline not being functorial, and we do not know if the persistent Hochschild characteristic defined this way is stable in the sense of Theorem~\ref{thm:stability}; this leaves open the following question:
 
\begin{question}
Is it possible to modify the composition of Equation~\eqref{eq:compvar} in a functorial way? Is the composition of Equation~\eqref{eq:compvar} stable in the sense of Theorem~\ref{thm:stability}?
\end{question}

\bibliographystyle{plain}
\bibliography{PersistentHH}

\begin{thebibliography}{10}

\bibitem{Aharoni}
R.~Aharoni, E.~Berger, and R.~Meshulam.
\newblock Eigenvalues and homology of flag complexes and vector representations
  of graphs.
\newblock {\em Geometric and Functional Analysis}, 15:555–566, 2005.

\bibitem{Atkin1972}
R.~Atkin.
\newblock From cohomology in physics to \(q\)-connectivity in social science.
\newblock {\em International Journal of Man-Machine Studies}, 4:139--167, 1972.

\bibitem{Atkin1974_book}
R.~Atkin.
\newblock {\em Mathematical Structure In Human Affairs}.
\newblock Heinemann, London, 1974.

\bibitem{Barmak_book}
J.~Barmak.
\newblock {\em Algebraic Topology of Finite Topological Spaces and
  Applications}.
\newblock Springer, 2011.

\bibitem{BAUES1985187}
H.-J. Baues and G.~Wirsching.
\newblock Cohomology of small categories.
\newblock {\em Journal of Pure and Applied Algebra}, 38(2):187--211, 1985.

\bibitem{angio}
P.~Bendich, J.~Marron, E.~Miller, A.~Pieloch, and S.~Skwerer.
\newblock Persistent homology analysis of brain artery trees.
\newblock {\em The Annals of Applied Statistics}, 10:198--218, 2014.

\bibitem{10.5555/320176.320203}
B.~Berger and P.~W. Shor.
\newblock Approximation alogorithms for the maximum acyclic subgraph problem.
\newblock In {\em Proceedings of the First Annual ACM-SIAM Symposium on
  Discrete Algorithms}, SODA '90, page 236–243, USA, 1990. Society for
  Industrial and Applied Mathematics.

\bibitem{vertechi}
M.~G. Bergomi and P.~Vertechi.
\newblock Rank-based persistence.
\newblock {\em Theory and Applications of Categories}, 35:228--260, 2020.

\bibitem{Brion}
M.~Brion.
\newblock Representations of quivers.
\newblock Available at
  \url{https://www-fourier.ujf-grenoble.fr/~mbrion/notes_quivers_rev.pdf}.

\bibitem{Brodzki}
J.~Brodzki, F.~Belchí, R.~Djukanovic, J.~Conway, M.~Pirashvili, and
  M.~Bennett.
\newblock Lung topology characteristics in patients with chronic obstructive
  pulmonary disease.
\newblock {\em Scientific Reports}, 8, 5341, 2018.

\bibitem{Bubenik_metrics}
P.~Bubenik, V.~de~Silva, and J.~Scott.
\newblock Metrics for generalized persistence modules.
\newblock {\em Foundations of Computational Mathematics}, 15:1501--1531, 2015.

\bibitem{caputi2022categorifying}
L.~Caputi, D.~Celoria, and C.~Collari.
\newblock Categorifying connected domination via graph \"{u}berhomology.
\newblock 2022.
\newblock ArXiv:2201.00721.

\bibitem{monotone}
L.~Caputi, D.~Celoria, and C.~Collari.
\newblock Monotone cohomologies and oriented matchings.
\newblock 2022.
\newblock ArXiv:2203.03476.

\bibitem{secondo}
L.~Caputi, C.~Collari, and S.~Di~Trani.
\newblock Combinatorial and topological aspects of path posets, and multipath
  cohomology, 2021.
\newblock Arxiv:2110.11206.

\bibitem{caputi2021multipath}
L.~Caputi, C.~Collari, and S.~Di Trani.
\newblock Multipath cohomology of directed graphs, 2021.
\newblock Arxiv:2108.02690.

\bibitem{neuro}
L.~Caputi, A.~Pidnebesna, and J.~Hlinka.
\newblock Promises and pitfalls of topological data analysis for brain
  connectivity analysis.
\newblock {\em NeuroImage}, 238:118245, 2021.

\bibitem{chaplin2021betti}
T.~Chaplin.
\newblock First betti number of the path homology of random directed graphs,
  2021.
\newblock ArXiv:2111.13493.

\bibitem{CHEN2001153}
B.~Chen, S.-T. Yau, and Y.-N. Yeh.
\newblock Graph homotopy and graham homotopy.
\newblock {\em Discrete Mathematics}, 241(1):153--170, 2001.
\newblock Selected Papers in honor of Helge Tverberg.

\bibitem{inbook_Chowdhury}
S.~Chowdhury and F.~Mémoli.
\newblock Persistent path homology of directed networks.
\newblock {\em Proceedings of the 2018 Annual ACM-SIAM Symposium on Discrete
  Algorithms}, pages 1152--1169, 2018.

\bibitem{chung}
M.K. Chung, J.L. Hanson, J.~Ye, R.J. Davidson, and S.D. Pollak.
\newblock Persistent homology in sparse regression and its application to brain
  morphometry.
\newblock {\em IEEE Transactions on Medical Imaging}, 34:1928--1939, 2014.

\bibitem{citterio}
M.~G. Citterio.
\newblock Classifying spaces of categories and term rewriting.
\newblock {\em Theory and Applications of Categories}, 9(5):92--105, 2001.

\bibitem{stabilityph}
D.~Cohen-Steiner, H.~Edelsbrunner, and J.~Harer.
\newblock Stability of persistence diagrams.
\newblock {\em Discrete and Computational Geometry}, 37:130--120, 2007.

\bibitem{Tribes_math}
P.~Concei\c{c}\(\tilde{\text{a}}\)o, D.~Govc, J.~Lazovskis, R.~Levi,
  H.~Riihim\"{a}ki, and J.~Smith.
\newblock An application of neighbourhoods in digraphs to the classification of
  binary dynamics.
\newblock {\em To appear in Network Neuroscience}, 2021.

\bibitem{dey2020efficient}
T.~K. Dey, T.~Li, and Y.~Wang.
\newblock An efficient algorithm for $1$-dimensional (persistent) path
  homology, 2020.
\newblock ArXiv:2001.09549.

\bibitem{vmv.20171272}
T.~K. Dey, S.~Mandal, and W.~Varcho.
\newblock {Improved Image Classification using Topological Persistence}.
\newblock In Matthias Hullin, Reinhard Klein, Thomas Schultz, and Angela Yao,
  editors, {\em Vision, Modeling and Visualization}. The Eurographics
  Association, 2017.

\bibitem{grdiestel}
R.~Diestel.
\newblock {\em Graph theory}.
\newblock Springer Heidelberg, 2010.

\bibitem{10.1007/978-3-662-47666-6_14}
J.~Dubut, {\'E}.~Goubault, and J.~Goubault-Larrecq.
\newblock Natural homology.
\newblock In M.~M. Halld{\'o}rsson, K.~Iwama, N.~Kobayashi, and B.~Speckmann,
  editors, {\em Automata, Languages, and Programming}, pages 171--183. Springer
  Berlin Heidelberg, 2015.

\bibitem{DUNAEVA201613}
O.~Dunaeva, H.~Edelsbrunner, A.~Lukyanov, M.~Machin, D.~Malkova, R.~Kuvaev, and
  S.~Kashin.
\newblock The classification of endoscopy images with persistent homology.
\newblock {\em Pattern Recognition Letters}, 83:13 -- 22, 2016.

\bibitem{Gabriel1967CalculusOF}
P.~Gabriel and M.~Zisman.
\newblock {\em Calculus of Fractions and Homotopy Theory}.
\newblock Springer Berlin Heidelberg, 1967.

\bibitem{giansiracusa}
N.~Giansiracusa, R.~Giansiracusa, and C.~Moon.
\newblock Persistent homology machine learning for fingerprint classification.
\newblock {\em 18th IEEE International Conference On Machine Learning And
  Applications (ICMLA)}, pages 1219--1226, 2019.

\bibitem{gidea}
M.~Gidea.
\newblock Topology data analysis of critical transitions in financial networks.
\newblock {\em SSRN Electronic Journal}, 01 2017.

\bibitem{govc2021complexes}
D.~Govc, R.~Levi, and J.~P. Smith.
\newblock Complexes of tournaments, directionality filtrations and persistent
  homology.
\newblock {\em Journal of Applied and Computational Topology}, 5:313--337,
  2021.

\bibitem{GradySchecnfisch}
R.~Grady and A.~Schenfisch.
\newblock Zig-zag modules: Cosheaves and {K}-theory.
\newblock 2021.
\newblock ArXiv:2110.04591.

\bibitem{Grigorianstax}
A.~Grigorian, R.~Jimenez, Y.~Muranov, and S.-T. Yau.
\newblock On the path homology theory of digraphs and eilenberg–steenrod
  axioms.
\newblock {\em Homology, Homotopy and Applications}, 20:179--205, 01 2018.

\bibitem{Grigoriancoho}
A.~Grigorian, Y.~Lin, Y.~Muranov, and S.-T. Yau.
\newblock Cohomology of digraphs and (undirected) graphs.
\newblock {\em Asian Journal of Mathematics}, 19:887--932, 11 2015.

\bibitem{grigoryan2013homologies}
A.~Grigor'yan, Y.~Lin, Y.~Muranov, and S.-T. Yau.
\newblock Homologies of path complexes and digraphs, 2013.
\newblock ArXiv:1207.2834.

\bibitem{grigoryan2014homotopy}
A.~Grigor'yan, Y.~Lin, Y.~Muranov, and S.-T. Yau.
\newblock Homotopy theory for digraphs, 2014.
\newblock ArXiv:1407.0234.

\bibitem{Grigoryan_first}
A.~Grigor’yan, Yong Lin, Y.~Muranov, and S.-T. Yau.
\newblock Path complexes and their homologies.
\newblock {\em Journal of Mathematical Sciences}, 248:564–599, 2020.

\bibitem{Grigoryan}
A.~Grigor’yan, Y.~Muranov, V.~Vershinin, and S.-T. Yau.
\newblock Path homology theory of multigraphs and quivers.
\newblock {\em Forum Mathematicum}, 30(5):1319--1337, 2018.

\bibitem{grigor2017homologies}
A.~Grigor’yan, Y.~Muranov, and S.-T. Yau.
\newblock Homologies of digraphs and k{\"u}nneth formulas.
\newblock {\em Communications in Analysis and Geometry}, 25(5):969--1018, 2017.

\bibitem{happel}
D.~Happel.
\newblock Hochschild cohomology of finite-dimensional algebras.
\newblock {\em Lecture Notes in Math}, 1404:108--126, 1989.

\bibitem{Harary1960SomePO}
F.~Harary and R.~Z. Norman.
\newblock Some properties of line digraphs.
\newblock {\em Rendiconti del Circolo Matematico di Palermo}, 9:161--168, 1960.

\bibitem{hatcher}
A.~Hatcher.
\newblock {\em {Algebraic topology}}.
\newblock Cambridge University Press, 2000.

\bibitem{MR11076}
G.~Hochschild.
\newblock On the cohomology groups of an associative algebra.
\newblock {\em Annals of Mathematics}, 46:58--67, 1945.

\bibitem{IVASHCHENKO1994159}
A.~V. Ivashchenko.
\newblock Contractible transformations do not change the homology groups of
  graphs.
\newblock {\em Discrete Mathematics}, 126(1):159--170, 1994.

\bibitem{KASSEL1987195}
C.~Kassel.
\newblock Cyclic homology, comodules, and mixed complexes.
\newblock {\em Journal of Algebra}, 107(1):195--216, 1987.

\bibitem{kaul}
M.~{Kaul} and D.~{Tamaki}.
\newblock {A Weighted Quiver Kernel using Functor Homology}.
\newblock {\em arXiv e-prints}, 2020.
\newblock ArXiv:2009.12928.

\bibitem{khalid}
A.~Khalid, B.S. Kim, M.K. Chung, J.C. Ye, and D.~Jeon.
\newblock Tracing the evolution of multi-scale functional networks in a mouse
  model of depression using persistent brain network homology.
\newblock {\em NeuroImage}, 101, 2014.

\bibitem{krishnan2014flowcut}
S.~Krishnan.
\newblock Flow-cut dualities for sheaves on graphs, 2014.
\newblock ArXiv:1409.6712.

\bibitem{kuang}
L.~Kuang, D.~Zhao, J.~Xing, Z.~Chen, F.~Xiong, and X.~Han.
\newblock Metabolic brain network analysis of fdg-pet in alzheimer’s disease
  using kernel-based persistent features.
\newblock {\em Molecules}, 24:2301, 2019.

\bibitem{lee-discriminativePH}
H.~Lee, M.K. Chung, H.~Kang, B.-N. Kim, and D.S. Lee.
\newblock Discriminative persistent homology of brain networks.
\newblock {\em Proceedings - International Symposium on Biomedical Imaging},
  pages 841--844, 03 2011.

\bibitem{li2020geometric}
F.~Li and B.~Yu.
\newblock The geometric realization of regular path complexes via
  (co-)homology, 2020.
\newblock Arxiv:1808.05337.

\bibitem{loday}
J.-L. Loday.
\newblock {\em Cyclic Homology}.
\newblock Springer Berlin Heidelberg, 1998.

\bibitem{L_tgehetmann_2020}
D.~Lütgehetmann, D.~Govc, J.~P. Smith, and R.~Levi.
\newblock Computing persistent homology of directed flag complexes.
\newblock {\em Algorithms}, 13(1):19, 2020.

\bibitem{maclane:71}
S.~MacLane.
\newblock {\em Categories for the Working Mathematician}.
\newblock Springer-Verlag, 1971.

\bibitem{masulli-villa}
P.~Masulli and A.E.P. Villa.
\newblock The topology of the directed clique complex as a network invariant.
\newblock {\em SpringerPlus}, 5(388), 2016.

\bibitem{munkres}
J.~R. Munkres.
\newblock {\em Elements of Algebraic Topology}.
\newblock Addison Wesley, 1984.

\bibitem{clique_communities}
G.~Palla, I.~Der\'{e}nyi, I.~Farkas, and T.~Vicsek.
\newblock Uncovering the overlapping community structure of complex networks in
  nature and society.
\newblock {\em Nature}, 435:814--818, 2005.

\bibitem{quillenI}
D.~Quillen.
\newblock Higher algebraic {K}-theory: I.
\newblock In H.~Bass, editor, {\em Higher K-Theories}, pages 85--147, Berlin,
  Heidelberg, 1973. Springer Berlin Heidelberg.

\bibitem{redondo}
J.~Redondo.
\newblock Hochschild cohomology: some methods for computations.
\newblock {\em Resenhas do Instituto do Matemática e Estatística da
  Universidade de São Paulo}, 2, 2001.

\bibitem{Reimann_2017}
M.~W. Reimann, M.~Nolte, M.~Scolamiero, K.~Turner, R.~Perin, G.~Chindemi,
  P.~Dłotko, R.~Levi, K.~Hess, and H.~Markram.
\newblock Cliques of neurons bound into cavities provide a missing link between
  structure and function.
\newblock {\em Frontiers in Computational Neuroscience}, 11, 2017.

\bibitem{ren2021differential}
S.~Ren and C.~Wang.
\newblock Differential algebras on digraphs and generalized path homology.
\newblock 2021.
\newblock ArXiv:2103.15870.

\bibitem{Riihimaki_simplicial_connectivities}
H.~Riihim\"{a}ki.
\newblock Simplicial \(q\)-connectivity of directed graphs with applications to
  network analysis.
\newblock 2021.
\newblock ArXiv:2202.07307.

\bibitem{Schroeder_book}
B.~Schröder.
\newblock {\em Ordered Sets, 2ed.}
\newblock Birkhäuser, 2016.

\bibitem{turner}
P.~Turner and E.~Wagner.
\newblock The homology of digraphs as a generalisation of {H}ochschild
  homology.
\newblock {\em Journal of Algebra and Its Applications}, 11:1250031, 2012.

\bibitem{West}
D.~B. {West}.
\newblock {\em {Introduction to graph theory}}.
\newblock Prentice-Hall, 2005.

\bibitem{wong}
E.~Wong, S.~Palande, B.~Wang, B.~Zielinski, J.~Anderson, and P.~Fletcher.
\newblock Kernel partial least squares regression for relating functional brain
  network topology to clinical measures of behavior.
\newblock {\em IEEE 13th International Symposium on Biomedical Imaging (ISBI)},
  pages 1303--1306, 04 2016.

\end{thebibliography}

\end{document}